\documentclass[graybox]{svmult}

\usepackage{type1cm}        
\usepackage{makeidx}         
\usepackage[all]{xy}

\usepackage{multicol}        
\usepackage[bottom]{footmisc}

\usepackage{newtxtext}       %
\usepackage{newtxmath}       

\usepackage{todonotes}
\usepackage{paralist}
\usepackage{float}
\usepackage{tikz}
\usetikzlibrary{matrix}
\usetikzlibrary{cd}
\usetikzlibrary{positioning}
\usetikzlibrary{decorations.markings}
\usetikzlibrary{math}
\usetikzlibrary{
  knots,
  hobby,
  decorations.pathreplacing,
  shapes.geometric,
  calc
}

\tikzset{
  knot diagram/every strand/.append style={
    ultra thick,
    red
  },
  show curve controls/.style={
    postaction=decorate,
    decoration={show path construction,
      curveto code={
        \draw [blue, dashed]
        (\tikzinputsegmentfirst) -- (\tikzinputsegmentsupporta)
        node [at end, draw, solid, red, inner sep=2pt]{};
        \draw [blue, dashed]
        (\tikzinputsegmentsupportb) -- (\tikzinputsegmentlast)
        node [at start, draw, solid, red, inner sep=2pt]{}
        node [at end, fill, blue, ellipse, inner sep=2pt]{}
        ;
      }
    }
  },
  show curve endpoints/.style={
    postaction=decorate,
    decoration={show path construction,
      curveto code={
        \node [fill, blue, ellipse, inner sep=2pt] at (\tikzinputsegmentlast) {}
        ;
      }
    }
  }
}

\usepackage{pgf}
\usepackage{pgfkeys}
\usepackage{pgfplots}
\pgfplotsset{width=7cm,compat=1.8}
\usepackage[backref=page,colorlinks=true]{hyperref}
\usepackage{bookmark} 

\makeatletter
\providecommand*{\toclevel@title}{0} 
\edef\toclevel@author{\the\numexpr\toclevel@title+1} 
\makeatother

\makeatletter
\providecommand*{\toclevel@title}{0}
\def\toclevel@author{1000}
\makeatother

\newenvironment{equationth}{\stepcounter{theorem}\begin{equation}}{\end{equation}}

\makeindex             

\newcommand{\pn}{{\mathbb{P}^{n}}}
\newcommand{\R}{\mathbb{R}}
\newcommand{\C}{\mathbb{C}}

\newcommand{\Z}{\mathbb{Z}}

\newcommand{\0}{\underline{0}}
\newcommand{\B}{\mathbb{B}}
\newcommand{\s}{\mathbb{S}}

\newcommand{\e}{\varepsilon}

\newcommand{\sF}{\mathcal{F}}
\newcommand{\F}{\mathcal{F}}

\newcommand{\GSV}{{\rm GSV}}
\newcommand{\CS}{{\rm CS}}

\newcommand{\Res}{{\rm Res}}
\newcommand{\BB}{{\rm BB}}

\spnewtheorem{theo}{Theorem}[section]{\bf}{\it}
\spnewtheorem{prop}[theo]{Proposition}{\bf}{\it}
\spnewtheorem{cor}[theo]{Corollary}{\bf}{\it}
\spnewtheorem{lem}[theo]{Lemma}{\bf}{\it}

\spnewtheorem{defin}[theo]{Definition}{\bf}{}
\spnewtheorem{proper}{Property}{\bf}{}
\spnewtheorem{quest}[theo]{Question}{\bf}{}
\spnewtheorem{probl}[theo]{Problem}{\bf}{}
\spnewtheorem{exa}[theo]{Example}{\bf}{}
\spnewtheorem{cond}[theo]{Condition}{\bf}{}

\def\Z{{\mathbb Z}}
\def\R{{\mathbb R}}

\def\C{{\mathbb C}}

\def\B{{\mathbb B}}
 
\def\F{\mathcal F}

\def\0{\underline 0}

\def\e{\varepsilon}

\def\O{\mathcal O}

\def\nn {\vskip 0.3cm \noindent }

\def\nn {\vskip.2cm \noindent}

\def\vv {\vskip.2cm}

\def\s{\mathbb{S}}
\def\mK{{\bf K}}
\def\ind{{\rm Ind}}

\motto{$\qquad \qquad $ In memory of Marco Brunella (1964-2012),  \\ $ \qquad \qquad$ whose mathematical legacy will endure forever.}

\begin{document}

\title*{Indices and residues: from Poincaré-Hopf to Baum-Bott, and Marco Brunella}
\author{Maur\'icio Corr\^ea and Jos\'e Seade}
\institute{Maur\'icio Corr\^ea \at  Universit\`a degli Studi di Bari, 
Via E. Orabona 4, I-70125, Bari, Italy, \email{mauricio.barros@uniba.it, mauriciomatufmg@gmail.com }
\and Jos\'e Seade \at Instituto de Matem\'aticas, Unidad Cuernavaca, Universidad Nacional Aut\'onoma de M\'exico. Ciudad Universitaria,  04510 Mexico City; 
\email{jseade@im.unam.mx}}
%
%
\maketitle

\abstract*{We study and discuss invariants of vector fields and holomorphic foliations that intertwine the theories of  complex analytic singular varieties and singular holomorphic foliations on complex manifolds:  two different settings with many points in common. }

\abstract{We study and discuss invariants of vector fields and holomorphic foliations that intertwine the theories of  complex analytic singular varieties and singular holomorphic foliations on complex manifolds:  two different settings with many points in common. 
}

\section*{Contents}
\setcounter{minitocdepth}{1}
\dominitoc

\section*{Introduction}\label{sec:1}
\mtaddtocont{\protect\contentsline{mtchap}{Introduction}{\thepage}\hyperhrefextend}

Holomorphic vector fields on complex manifolds determine   holomorphic  foliations   by complex curves with singularities at the zeros -or singularities- of the vector field.  
The leaves of this foliation are immersed Riemann surfaces. 
The most basic invariant of a vector field at an isolated singularity is its Poincaré-Hopf  local  index, 
and the Poincaré-Hopf index theorem can be regarded as saying that the Euler characteristic of the manifold, which is a global  topological invariant, actually is localized at the singularities of the vector field. A myriad of important  lines of research spring from this theorem, permeating all mathematics. Two of them are particularly relevant for us: 
\begin{itemize}
\item The theory of characteristic classes of vector bundles and singular varieties.  This is much related with indices of vector fields and frames; and
\item The analogous theory for holomorphic singular  foliations on manifolds. 
\end{itemize}

Indeed, the foundational results by R. Bott \cite{Bott1967} on indices and residues for holomorphic vector fields with singularities (not necessarily isolated) have profoundly influenced areas ranging from foliation theory to   applications in enumerative algebraic geometry \cite{EllingsrudStromme1996,Kontsevich1995}. This latter field was notably advanced by Kontsevich \cite{Kontsevich1995}, who provided a rigorous framework for formulating and verifying predictions in the enumerative geometry of curves, as derived from Mirror Symmetry. In \cite{Baum1} Baum and Bott subsequently generalized  Bott's theorem to meromorphic vector fields, and later  extended these for higher-dimensional foliations \cite{Baum2}. Whence, we see that the indices  and residues associated with foliations are of great significance, and the work of M. Brunella, T. Suwa, M. Soares and others turned this into a rich and fascinating area of mathematics where a lot is happening nowadays. See for instance \cite{Suwa-Herman} for an account on this subject.

On the other hand the astounding impact of characteristic classes in manifolds theory of course inspired the search for such invariants  in the setting of singular varieties. The work of D. P. Sullivan, M.-H. Schwartz, R. MacPherson, W. Fulton, P. Baum and others laid down  the foundations of the theory, and then the work of J.-P. Brasselet, J. Schürmann, L. Maxim, P. Aluffi  and others are making this a rich and fascinating theory as well.  And as noted for instance in \cite{BLS, BSS, EG, CMS-2}, the theory of indices of vector fields and 1-forms comes into the picture playing an important role.

In fact the study  of indices and residues
  is  of utmost importance in the local and global  study of both,  singular holomorphic foliations in complex manifolds and singular complex analytic varieties: two different theories with many  points in  common. The purpose of this work is to thread a path 
   intertwining  these two theories. 
  We use as spearhead  the GSV-index.

The first indices of vector fields and 1-forms on singular varieties in the literature appeared in the work of M. H. Schwartz \cite {Sch1} and R. MacPherson \cite{MP}. Yet, this was  only for specific vector fields or 1-forms, aimed to defining Chern classes for singular varieties. The GSV index   was introduced in \cite{Seade, GSV1} for   vector fields on isolated complex hypersurface singularities,  and it began the theory of  indices of vector fields and 1-forms on singular varieties.

The definition of the GSV index in 1987 was followed by the radial index (1991, and then again by different authors  in 1997, 1998), the homological index (1998),  indices for 1-forms (2001), the local Euler obstruction for vector fields (2004) and for 1-forms (2005), the logarithmic index (2005) and several other related invariants. We briefly discuss these invariants in Section 1; we also comment on their relation with Chern classes for singular varieties, and we refer the reader to \cite{BSS, CMS-2} for more on this subject.

From the viewpoint of foliations and residues, in his  remarkable paper \cite {Bru}, M. Brunella gave in 1997  an interpretation of the GSV index in the setting of holomorphic 1-dimensional foliations in complex surfaces, showing that this index is of great importance in the theory of holomorphic foliations.  In fact 
several of the invariants that we study in the sequel  spring from the  Baum-Bott residues, so we briefly discuss these  in Section 2. This section also serves as an introduction to the subject of singular holomorphic foliations and Pfaffian systems.

One of the main problems in this field is the famous Poincaré problem, which seeks to bound the degree of invariant varieties of holomorphic foliations in terms of the degree of the foliation. In   \cite{Bru}, Brunella showed that the non-negativity of the total sum of the GSV index of the foliation along an invariant curve serves as an obstruction to affirmatively solving the Poincaré problem. He further proved in his book \cite{Brunella2004} that this invariant plays a significant role for the classification of holomorphic foliations in complex surfaces.

  Section \ref{S. Brunella} has  Brunella’s article \cite {Bru} as its core; we focus here on 1-dimensional holomorphic foliations on complex surfaces. We briefly discuss the Baum-Bott residues in this setting, and their relations with the Milnor number for foliations, the Camacho-Sad indices, the variations defined by Khanedani and Suwa, and the GSV index as defined by Brunella. We recall Brunella’s theorem saying that 
  if we have a holomorphic foliation $\F$ is a complex surface $X$ and 
  $S$ is a nondicritical separatrix at a point $p$ then:
$$\mathrm{GSV}(\F,S,p) \ge 0 \,.$$ 

If $\F$ further is a generalized curve ({\it i.e.} there are no saddle-nodes in its resolution) and $S$  is maximal,  union of all separatrices at $p$, then:
$$\mathrm{GSV}(\F,S,p) = 0 \,.$$
And the reciprocal from \cite{CavalierLehmann2001}:  if $\mathrm{GSV}(\F, S, p) = 0 \,,$ then  $\F$ is a generalized curve at $p$ and $S$ is the maximal separatrix.

Sections \ref {Poincare Problem} and \ref{Log-Alex}  look at higher dimensions.   These  are mostly based on work by Corr\^ea and Machado \cite{CorreaMachado2019, Correa-MachadoGSV} and  Corr\^ea, Louren\c co and   Machado  \cite{CLM2024}.  
  There are two previous articles that somehow paved the way. 
  In \cite{SS1}, Seade and Suwa provide Baum-Bott type formulas for foliations by curves induced by global holomorphic vector fields on complex manifolds with a non-empty boundary.  Also, a  Bott-type formula for complex orbifolds is studied in \cite{CorreaRodriguezSoares2016}, and formulas relating residues via good resolutions are obtained.  
       To address the case of a  projective variety $Y$ with more complicated  singularities, it is natural to consider a log resolution $(X, D) \to Y$ and to relate the residues of holomorphic vector fields on $X$ to the Chern classes of the logarithmic tangent bundle $TX(-\log D)$. This is   explained  in Section \ref{Log-Alex}. 
    
  In Section \ref {Poincare Problem}, using 
Aleksandrov's Decomposition Theorem we define, following \cite{Correa-MachadoGSV}, the GSV index for 
Pfaff systems $\omega$ of dimension $k$ on complex 
manifolds $X$, relative to an invariant variety  $V$ in  $X$ which is  a reduced local complete intersection of dimension $n-k$. This is done {\it à la} Brunella and associates  an index to each  codimension $1$ component  $S \subset V$ of the singular set of the Pfaff system 
induced by a twisted form $\omega \in \mathrm{H}^0(X, \Omega_X^{n-k} \otimes \mathcal{L})$. 
It is proved that   these indices localize the Chern class of the line bundle $\mathcal{L} \otimes \det (N_{V/X})^{-1}$ restricted to $V$. Then  a method  is given for calculating these GSV-indices, which can be compared with  formulas by Suwa \cite[Proposition 5.1]{Suwa2014} and Gómez-Mont \cite{Gom}. This partially addresses Suwa's question in \cite[Remark 5.3.(6)]{Suwa2014}. As
an application, one has  that the non-negativity of the GSV-index gives 
an obstruction to the solution of the Poincar\'e problem for Pfaff systems, in concordance with Brunella’s observation for  holomorphic foliations on surfaces. This further yields to  a
bound for the Poincaré problem similar to those previously established by Esteves and Cruz \cite{EstCruz} and by Corrêa and Jardim \cite{CJ}. Related results were previously obtained in  \cite{CLins, SoaresPoinc1, CavalierLehmann2006, BruMe}. In particular, this 
recovers Soares' bound in the case where the invariant hypersurface $V$ is smooth.

     In section  \ref{ss. Suwa} we recall Suwa’s  interpretation in  \cite{Suwa2014}  of the GSV and virtual indices for vector fields as residues on complete intersection varieties with isolated singularities, and the fact that if   $V = \{ f_1 = \dots = f_k = 0 \}$ is a local complete intersection in $\C^n$, with isolated singularity at $0$,  invariant by a foliation by curves $\F$ singular at most at $0$,  then the index is a localization of the top Chern class:
\[
\mathrm{GSV}(\F,V, 0) = \operatorname{Res}_{c_n}(\F, T U|_V ; 0) = \mathrm{Vir}(\F,V, 0). 
\]

   Section  \ref{Log-Alex} recalls  Aleksandrov’s logarithmic  index  $\operatorname{{Ind}_{\log}}(\mathcal{F}, D, p) $ defined in \cite{Alex}, inspired by the work of G\'omez-Mont, that measures the variation between the Poincar\'e--Hopf and  the homological indices. Then, using this index,  we address (following \cite{CorreaMachado2019, CorreaMachado2024})  the problem of establishing a global logarithmic residue theorem for meromorphic vector fields on compact complex manifolds.

    One has that if 
    $\mathcal{F}$ is a one-dimensional foliation with isolated singularities and logarithmic along  a normal crossings divisor $D$ in a compact complex manifold $X$. Then
$$
\int_{X} c_{n}(T_{X}(-\log D) - T\mathcal{F}) = \sum_{x \in \operatorname{Sing}(\mathcal{F}) \cap (X \setminus D)} \mu_x(\mathcal{F}) + \sum_{x \in \operatorname{Sing}(\mathcal{F}) \cap D} \operatorname{Ind}_{\log}(\mathcal{F}, D, x),
$$
where  $\mu_x(\mathcal{F})$ is the Milnor number of $\mathcal{F}$ at $x$. As an application, in \cite{CorreaMachado2019}, the authors prove that if $\mathcal{F}$  is  a one-dimensional foliation on $\mathbb{P}^n$, with $n$ odd, and  isolated singularities  (non-degenerate)  along a  smooth divisor $D$, then  $\mbox{Sing}(\F)\subset D$ if and only if  Soares’ bound \cite{SoaresPoinc1} for
the Poincar\'e problem is achieved, i.e, if and only if
$$
\deg(D)=\deg(\F)+1.
$$
Recently, in  \cite{CLM2024}, a logarithmic version of the Baum--Bott theorem was obtained for more general symmetric polynomials.

In the last section we briefly discuss relations with the adjunction formula and with other invariants, such as the Milnor and Tjurina  numbers for foliations.

\section{Indices of vector fields on  isolated complete intersection  singularities}

  Throughout this work, by 
  a
complete intersection singularity   in $\C^{n+k}$ we mean the germ (say at $0$) of a geometric --non necessarily algebraic-- complete intersection $(V,0)$, defined by as many holomorphic equations as its codimension, say $k$. If 
$V$  is defined by 
$$ f = (f_1,..., f_k)  \, \colon \, (\B_\e,0)  \longrightarrow \, ({\C^k},0) \, ,
$$
and  the critical set of $f$ meets $V$  only at $0$, then we say that $(V,0)$ is an isolated complete intersection singularity, an ICIS for short.

A vector field $v$ on $V$ is a continuous section of the bundle $T\C^{n+1}|_V$  that vanishes at $0$ and 
  for each $x \in V^* = V \setminus 0$ the vector $v(x)$ is tangent to $V^*$. The vector field is holomorphic if the corresponding section is holomorphic. Notice that $v$ must vanish at $0$ because this is an isolated singularity of $V$. 

Notice that for each $x \in V^* = V
\setminus \{0\}$, the tangent space $T_xV^*$ consists of all
vectors in $T_x\C^{n+1}$ which are mapped to $0$ by the differential of all the $f_i$: 
$$T_xV^* \,=\,\{ \zeta \in T_x\C^{n+1}\,\big|\, df_i(x)(\zeta) = 0 \,, \quad \forall \; i=1,...,k\, \}\,.$$

\begin{example}\label{hamilton}
If $V$ is defined in $\C^{2n}$ by a map $f:(\C^{2n},0)
\to (\C,0)$  and we set $V= f^{-1}(0)$, then the Hamiltonian vector field $$\widetilde
\zeta(z_1, ..., z_{2n}) = \Big(- \frac{\partial \,f}{\partial z_2},
\frac{\partial\,f}{\partial z_1}, ...,- \frac{\partial \,f}{\partial z_{2n}}, \frac{\partial \,f}{\partial z_{2n-1}} \Big)$$
 is tangent to $V$ and it is
zero only at the origin. Notice that this vector field is actually
tangent to all the fibers $f^{-1}(t)$. \end{example}

\begin{example}\label{homogeneous}
Now let $V$ be defined by the  homogeneous polynomial   map $f:(\C^{n},0)
\to (\C,0)$ defined by  $(z_1,...,z_n) \mapsto z_1^k + ... + z_n^k$ for $k \ge 2$. Then the radial vector field $(z_1,...,z_n) = (z_1,...,z_n)$ is tangent to $V$ and we will see later that, unlike the previous example, if $k \ge 3$ this can never be extended to the ambient space with an isolated singularity at $0$ and being tangent to the fibers $f^{-1}(t)$. The case $k=2$ is special (see Example \ref{ex. homogeneous2}.
\end{example}

\subsection{The GSV index}\index{GSV index !}

Given an  ICIS $(V,0)$ as above, in ${\C}^{n+k}$, defined by $ f = (f_1,..., f_k) $
we have that the 
(complex conjugate) gradient vector fields  $\, \overline{\nabla} f _1, ..., \,\overline{\nabla}f_k$ are linearly independent  everywhere and normal
to $V$ (for the usual hermitian metric in ${\C}^{n+k}$).  Hence if   $v$
is  a continuous vector field on $V$ singular only at $0$, then the set
$ \, \{v(z), \overline{\nabla} f _1, \dots , \,\overline{\nabla}f_k)\} \,$ is a $(k+1)$-frame at each point in $V^* :=
V\setminus \{0\}$, and up to homotopy it can be assumed to be
orthonormal. Hence these vector fields define a
continuous map from $V^*$ into the Stiefel manifold of complex
orthonormal $(k+1)$-frames in ${\C}^{n+k}$, denoted
$W_{k}(n+k)$.

Let $\,\mK = V \cap \s_{\varepsilon} \, $ be the link of $0$ in
$V$. It is an oriented, real manifold of dimension $(2n-1)$ and
the above frame defines a continuous map
$$ \phi_v \, = \, (v,\overline{\nabla} f _1, ..., \,\overline{\nabla}f_k) \,
\colon \, \mK \, \longrightarrow W_{k+1}(n+k) \, .$$
The Stiefel manifold $W_{k}(n+k)$ is diffeomorphic to the
homogeneous space $U(n+k)/U(n)$ and therefore  the homotopy
sequence associated to this fibration implies that
 $W_{k}(n+k)$ is  $(2n-2)$-connected,  while its homology in dimension $(2n-1)$  is
  isomorphic to $\Z$. Hence, if $\mK$ is connected, 
the map $ \phi_v$ has a well defined degree $\deg(\phi_v) \in \Z$,
defined by means of the induced homomorphism $H_{2n-1}(\mK) \to
H_{2n-1}(W_{2}(n+1))$ in the usual way.

\begin{definition}\index{GSV index !}\label{def GSV} Assume $n \ge 2$ or $n=1$ and $V$ is irreducible.
Then {\bf the  GSV index} of $v$ at $0 \in V$, $\ind_{\rm GSV}(v;V,0)$,  is the
degree of the above map $\phi_v$.
\end{definition}

\begin{figure}
	\centering
\quad \includegraphics[height=5cm ]{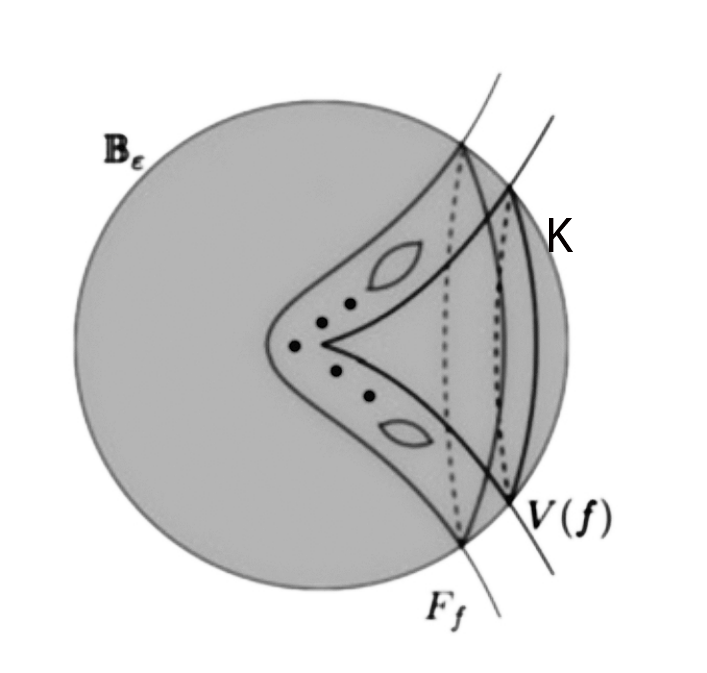}
	\caption{The  GSV equals the PH index  in a Milnor fiber.}
\end{figure}

\begin{remark} Note  that the link $\mK$ is always connected if $n \ge 2$. For $n=1$ the link has one connected component for each irreducible component of the plane curve $V$ and in this case one needs to take into account the intersection numbers of the branches. We refer to \cite[3.2]{BSS} and to Section \ref{S. Brunella} for more about this case, which is specially important for the theory of holomorphic foliations.  
\end{remark}

We recall that one has a Milnor   fibration
associated to the map $f$, see \cite {Mi1, Ham, Loo}, and the
Milnor fiber $F$ can be regarded as a compact $2n$-manifold with
boundary $\partial F = \mK$.  By the Ehresmann Fibration Lemma there is an ambient
isotopy of the sphere $\s_{\varepsilon}$ taking $\mK$ into the boundary
$\partial F$, which can be extended to a collar of $\mK$, which
goes into a collar of $\partial F$ in $F$.  Hence $v$ can be
regarded as a non-singular vector field  on $\partial F$.  

We also recall that the Milnor fiber $F$ has the homotopy type of a bouquet of spheres of middle dimension and therefore its Beti numbers are all $0$ except $b_0 = 1$ and $b_n = \mu$ for some $\mu \ge 0$

\begin{definition}\index{Milnor number !}\label{def Milnor number}
The number $b_n = \mu$ is {\bf the Milnor number }of $V$ at $0$.
\end{definition}

\begin{proposition}\label{GSV} The $\mathrm{GSV}$ index has the following properties:

\vv

\noindent
{\rm  (1)}  The $\mathrm{GSV}$ index of $v$ at $0$ equals the total 
Poincar\'e-Hopf index of $v$ in the Milnor fiber:  
$$\mathrm{GSV}(v;V,0) = PH(v;F) \, .$$

\vv

\noindent
{\rm (2)}  If $v$ is everywhere transverse to $\mK$, then
$$ \mathrm{GSV}(v;V,0) = 1 +(-1)^n \mu \, ,$$\index{Milnor number !}
where $n$ is the complex dimension of $V$ and $\mu$ is the Milnor
number.
\end{proposition}

The proof is easy and we refer to \cite[3.2]{BSS} for details. 

We notice that we know from \cite {Mi1} that if $V$ is a hypersurface germ defined by a map $f$, then its Milnor number $\mu $ is the Poincaré-Hopf local index of the gradient vector field $\nabla f$, and therefore it equals the intersection number,
\begin{equation}\label{alg Milnor number}\index{Milnor number !}
 \mu = {\rm dim}_\C \; \frac{{\mathcal O}_{\C^{n+1},0}}{\big(\frac{\partial f}{\partial z_0},..., \frac{\partial f}{\partial z_n}\big)} \;.
\end{equation}
For complete intersections the analogous expression is given by the Lê-Greuel formula \cite{Le7, Gr}. 

\begin{example}\label{ex. homogeneous2}
Consider  the radial vector field $v_{rad}$ in Example \ref{homogeneous} for $k=2$. The statement (2) in the  theorem above implies:
$$ \mathrm{GSV}(v_{rad};V,0) = 1 +(-1)^n \mu  = 1 +(-1)^n \,.\
$$
Hence if $n$ is odd, then the index is $0$ and it is an exercise to show that in this case the vector field can be extended continuously to the Milnor fibers with no singularities.
\end{example}

 When the vector field in question is holomorphic, the GSV-index has other interesting properties. We now discuss some  of these. 
 To begin, recall that if $\zeta = (\zeta_1,...,\zeta_n)$ is a holomorphic vector field in $\C^n$ with an isolated singularity at $0$, then its local Poincaré-Hopf index is the intersection number:
 \begin{equation}\label{PH-holomorphic}
\mathrm{PH}(v;0) \,= \, {\rm dim}_\C \; \frac{{\mathcal O}_{\C^{n+1},0}}{(\zeta_1,...,\zeta_n)} 
 \end{equation}
 
 Recall that if $\zeta$ is a holomorphic vector field on a complex variety  $Z$ (for simplicity, say with isolated singularities), then the 
 solutions of $\zeta$ are Riemann surfaces immersed in $X$, and one has an associated 1-dimensional holomorphic foliation  with singularities at the zeroes of $\zeta$.  

\begin{definition}\index{Topologically equivalent}
 Let  $(V_i,0)$, $i=1,2$, be  germs of isolated  hypersurface  singularities defined by holomorphic map-germs 
$(\C^{n+1},0)   \buildrel {f_i} \over {\to} (\C,0)$,  
 and 
let  $v_1$, $v_2$ be germs of holomorphic vector fields on the $V_i$ with an isolated singularity at $0$.  
Then we say that:
\begin{enumerate}
\item 
The germs $(V_1,0)$  and $(V_2, 0)$ are 
{\bf topologically equivalent}  if there exists an orientation preserving local
homeomorphism $h$ of $\C^{n+1}$, such that $f_2 = f_1 \circ h$. In this case we say that $h$ is a topological equivalence between $V_1$ and $V_2$. 
\item The holomorphic vector fields $v_1$, $v_2$ are {\bf topologically equivalent} if there exists a topological equivalence between $V_1$ and $V_2$ carrying the leaves of the foliation $\mathcal F_1$ of $v_1$ into the leaves  corresponding to $v_2$.
\end{enumerate}
\end{definition}

One has  \cite[Theorem 4.5]{GSV1}:  

\begin{theorem}\label{thm top invariance}
Let $(V,0)$  denote the germ of an isolated hypersurface singularity in
$\C^{n+1}$. Then, the $\mathrm{GSV}$-index of a holomorphic vector field   is a topological invariant. That is, if $\zeta$ is a holomorphic vector field on $(V,0)$ and it is topologically equivalent to another holomorphic vector field $\zeta’$  on some other hypersurface germ, then $\zeta$ and $\zeta’$ have the same $\mathrm{GSV}$-index.  
\end{theorem}

 We recall that if $\mathcal F$ is a 1-dimensional holomorphic foliation on a complex manifold, then $\mathcal F$ is locally free and therefore 
if  $p$ is an isolated singularity of the foliation, then $\mathcal F$ is locally generated by a holomorphic vector field. In analogy with equation \ref{alg Milnor number} one has:

\begin{definition}\index{Milnor number ! of a foliation}\label{Milnor number of a foliation}
Let $\zeta = (\zeta_1,...,\zeta_n)$ be a holomorphic vector field tanget to   $\mathcal F$ in a neighborhood of the isolated singularity $p$.  
Then  the  {\bf Milnor number} of   $\mathcal F$ at $p$ is the Poincaré-Hopf local index of $\zeta$ at $p$, so it equals:
$$\mu(\mathcal F,p) := {\rm dim}_\C \; \frac{{\mathcal O}_{\C^{n+1},0}}{(\zeta_1,...,\zeta_n)} \;.
$$
\end{definition}

It is easy to prove (see \cite{GSV1}) that this invariant equals the PH-local index of any continuous vector field tangent to  $\mathcal F$ near $p$. One  gets the following  corollary to Theorem \ref{thm top invariance}, which is Theorem A in \cite{CLS}.

\begin{theorem}\label{top inv mu}
 The Milnor number of an isolated singularity of a 1-dimensional holomorphic foliation on a smooth complex manifold, is a topological invariant. 
\end{theorem}

We now remark that, as noticed in \cite{B-GM}, given an isolated complex singularity germ $(V,0)$ in some $\C^m$ and a holomorphic vector field $\zeta$ on  $V$ with an isolated singularity at $0$, there are infinitely many extensions of 
$\zeta$ to a holomorphic vector field in a neighborhood of $0$ in $\C^{m}$. 
Yet, as noticed in \cite [Section 3]{GSV1}, if $(V,0)$ is defined by a single equation $f$
 in $\C^m$, and if   $\zeta$ is a holomorphic vector field on $V$ with an isolated singularity at  $0$, then an extension $\widetilde \zeta$ to the ambient space can have a singular locus of dimension at most $1$, because cutting down by a hypersurface can reduce the dimension at most by $1$. 

One has the following Theorem 3.2 from \cite {GSV1}.

\begin{theorem}\label{tangent to fibers}
Let $(V,0)$  denote a hypersurface singularity germ in
$\C^{n+1}$ defined by a function $f$ with at most an isolated singularity at $0$ (which could be a smooth point in $V$). Let $\zeta$ be a holomorphic vector field on $V$ with an isolated singularity at $0$. Suppose there exists a holomorphic extension $\widetilde \zeta$ of $\zeta$ to a neighborhood $U$ of $0$ in $\C^{n+1}$ such that $\widetilde \zeta$ is tangent to the fibers of the map $f$. Let $\Sigma$ be the zero-locus of $\widetilde \zeta$, which has dimension at most $1$. Then:

\begin{enumerate}
\item The $\mathrm{GSV}$-index of $\zeta$ in $V$ is the intersection number of  a Milnor fiber \\ $F_t = f^{-1}(t) \cap U$ and $\Sigma$. That is:
$$\mathrm{GSV} (\zeta;V,0) \, = \, \sum_{p \in U} \, {\rm dim}_\C \,
 \frac{{\mathcal O}_{\C^{n+1},p}}{(\widetilde \zeta_0,...,\widetilde \zeta_n, f-t)} \;,
$$
where the $\widetilde \zeta_0$ are the components of $\widetilde \zeta$. Hence, in this case the GSV index is non-negative and it is $0$ if and only if  $\widetilde \zeta$ has an isolated singularity at $0$. 

\item Localized at $0$ the formula becomes:
$$\mathrm{GSV}(\zeta;V,0) \, = \,  {\rm dim}_\C \,
 \frac{{\mathcal O}_{\C^{n+1},0}}{(\widetilde \zeta_0,...,\widetilde \zeta_n, f)} \; - \;  {\rm dim}_\C \,  \frac{(f) \cap (\widetilde \zeta_0,...,\widetilde \zeta_n, f)} {(f\,\widetilde \zeta_0,..., f \,\widetilde \zeta_n)}  \;.
$$

\end{enumerate}
\end{theorem}

Notice that the number ${\rm dim}_\C \,
 \frac{{\mathcal O}_{\C^{n+1},0}}{(\widetilde \zeta_0,...,\widetilde \zeta_n, f)} \; $
 is a reminiscent of the Tjurina number for functions (see Section \ref{s. final}) and can be called the Tjurina number for vector fields. This appears again in the next section,  in Gómez-Mont’s formula for the homological index.
 
For instance, coming back to Example \ref{homogeneous}, the Milnor number is $(k-1)^{n+1}$, by \cite{Mi1}. Hence, if $n$ is odd, then the GSV index of the radial vector field is negative and therefore it can never be holomorphically extended being tangent to the Milnor fibers, not even with singularities. 


\subsection{The homological index}
Theorem \ref{tangent to fibers} gives  an algebraic  formula for the GSV index of holomorphic vector fields under very stringent conditions. Gómez-Mont’s search for an algebraic formula  in the vein of (\ref {PH-holomorphic})  for the GSV index, opened an area of research which is of interest on its own; this all springs from the definition of a homological index in \cite{Gom}. This has given rise to a series of papers by Gómez-Mont and several co-authors
 (see \cite{B-GM, GGM1,GGM2,GGM3, BEG}), including the real analytic setting where it extends to singular hypersurfaces the celebrated index formula of Eisenbud-Levine \cite {EL} and Khimshiashvili \cite{Khi}. There is also a homological index for 1-forms  in \cite{EGS}, and 
Aleksandrov’s logarithmic index  for vector fields and 1-forms \cite{Alex, Alex2} that will be discussed in Section \ref{Log-Alex}. In \cite {Zach1, Zach2} M. Zach uses variations of the homological index to study interesting generalizations of the Milnor number  to the case of holomorphic maps on singular varieties.

The presentation below  is essentially extracted from \cite[Chapter 7]{BSS} and it is a {\it résumé} of part of Gómez-Mont’s work in  \cite{Gom}. 

We  recall that as noted above,  if $\zeta := (\zeta_1,..., \zeta_n)$ is a holomorphic vector field in $\C^n$ with an isolated singularity at $0$, then its Poincaré-Hopf local index is the intersection number:
$$\mathrm{PH}(\zeta; 0) \, = \, {\rm dim}_\C \, \frac{\mathcal O_{\C^n,0}}{(\zeta_1,..., \zeta_n)} \,.
$$
We would like to have a similar expression for the GSV index.  Yet we know from  Example \ref{ex. homogeneous2} that this index can be negative, even for holomorphic vector fields. Therefore  an algebraic expression for the GSV index cannot simply be given by the dimension of some algebra. One may notice that for vector fields in $\C^n$ one has the Koszul complex

\begin{equationth}\label{Koszul}
0 \longrightarrow \Omega^n_{\C^n,0}\longrightarrow \Omega^{n-1}_{\C^n,0}\longrightarrow \cdots\longrightarrow
\O_{\C^n,0}\longrightarrow 0\, ,
\end{equationth}\hskip -1pt
where the arrows are contraction by $\zeta$. The Euler characteristic $\chi$ of this complex is $PH(\zeta; 0)$: this is the clue for defining the homological index, though there is an interesting difference. In the smooth case, all the homology of the Koszul complex vanishes except in top dimension, which contributes to $\chi$  with the PH index. In the singular case one has a similar complex but the homology groups do not vanish. The miracle  is that “almost all of them” have the same dimension, so these cancel down when considering the Euler characteristic  (see \cite {Gom} for the precise statement).

Let $(V,0)\subset ({\C}^m, 0)$ be a germ of a complex analytic
variety of pure dimension $n$, which is regular on $V \setminus
\{0\}$. So $V$ is either regular at $0$ or else it has
  an isolated singular point at the origin.  A vector field $\zeta$
  on $(V,0)$ can always be defined   as the restriction to $V$ of a vector field
  $\widehat v$ in the
ambient space which is tangent to $V \setminus \{0\}$;
   $\zeta$ is holomorphic if $\widehat v$ can be chosen to be holomorphic. So we may write
   $\zeta = (a_1,\cdots, a_m)$ where the $a_i$ are restriction to $V$ of holomorphic functions on a
neighborhood of $0$ in $(\C^m, 0)$.
  
  It is worth noting that given any space $V$ as above, there are always holomorphic vector fields on
  $V$ with an isolated singularity at $0$. This   non-trivial fact is
  a weak form of
 a stronger result (\cite[p. 19]{B-GM}): in the space
$\Theta(V,0)$
  of germs of holomorphic vector fields
  on $V$ at $0$, those having an isolated singularity form a connected, dense open subset
  $\Theta_0(V,0)$.
   Essentially the same result implies also that every $\zeta \in \Theta_0(V,0)$ can be extended
  to a germ of a holomorphic vector field in $\C^m$ with an isolated singularity, though it can also be
 extended with a singular locus of dimension more that $0$, a fact that may be useful for explicit
computations (as in the sequel).

 A (germ of a) holomorphic $j$-form on $V$ at $0$
 means the restriction to $V$ of a holomorphic $j$-form on a neighborhood of
  $0$ in $\C^m$; two such forms in $\C^m$ are equivalent if their restrictions to $V$
  coincide on a neighborhood of $0 \in V$.
  We denote by $\Omega^j_{V,0} $ the space of all such forms (germs); these are the K\"ahler
differential
 forms on $V$ at $0$.
  So $\Omega^0_{V,0} $ is the local structure ring $\O_{(V,0)}$
  of holomorphic functions on $V$ at $0$ and each
  $\Omega^j_{V,0} $ is an $\Omega^0_{V,0} $-module. Notice that if the germ of $V$ at $0$ is
  determined by $(f_1,\cdots,f_k)$ then one has:
 \begin{equationth}\label{differential}
 \Omega^j_{V,0} := \, \Omega^j_{\C^m,0} / ( f_1\Omega^j_{\C^m,0} + df_1 \wedge
 \Omega^{j-1}_{\C^m,0}\,, \dots, \, f_k\Omega^j_{\C^m,0} + df_k \wedge
 \Omega^{j-1}_{\C^m,0}) \;,
 \end{equationth}
 \hskip -4pt
where $d$ is the exterior derivative.

Now, given a holomorphic vector field $\widehat v$ at $0 \in \C^m$
with an isolated singularity at the origin, and a differential
 form $\omega \in \Omega^j_{\C^m,0}$, we can always contract $\omega$ by $\zeta$ in the usual way,
 thus getting a differential  form $i_v(\omega) \in \Omega^{j-1}_{\C^m,0} $. If $\zeta = \widehat v\vert_V$
is tangent to $V$, then contraction is well defined at the level
of differential
 forms on $V$ at $0$ and one gets a
complex $(\Omega^\bullet_{V,0}, v)$:
\begin{equationth}\label{1.1}
0 \longrightarrow \Omega^n_{V,0}\longrightarrow \Omega^{n-1}_{V,0}\longrightarrow \cdots\longrightarrow
\O_{V,0}\longrightarrow 0\, ,
\end{equationth}\hskip -1pt
where the arrows are contraction by $\zeta$ and $n$ is the dimension
of $V$; of course one also has differential
 forms of degree $>n$, but those forms do not play a significant role here.
 We consider the homology groups of this complex:
$$
H_j(\Omega^\bullet_{V,0}, v) = {\rm Ker}\,(\Omega^{j}_{V,0} \to
\Omega^{j-1}_{V,0})/{\rm Im}\,(\Omega^{j+1}_{V,0} \to
\Omega^j_{V,0}) \,.
$$
The first observation in \cite {Gom} is that if $V$ is regular at
$0$, so that its germ at $0$ is
 that of $\C^n$ at the origin,
 and if $\zeta = (a_1,\cdots,a_n)$ has an isolated singularity
at $0$, then this is the usual Koszul complex. In that case,  its homology groups vanish
for $j > 0$, while
$$H_0(\Omega^\bullet_{V,0}, v)  \cong \O_{\C^n,0}\big / (a_1,\cdots,a_n)\,.$$
In particular the complex is exact if $\zeta(0) \ne 0$. Since the
contraction maps are $\O_{V,0}$-module maps, this implies that if
$V$ has an isolated singularity at the origin, then the homology
groups of this complex are concentrated at $0$, and they are
finite dimensional because the sheaves of K\"ahler differentials on $V$
are coherent. Hence it makes sense to define:

\begin{definition}\index{Homological index !}\label{def hom ind}
  The {\it homological index}
$\,{\rm Ind}_{\rm hom}(\zeta,0;V)$ of the holomorphic vector field $\zeta$
on $(V, 0)$ is  the Euler characteristic of the above complex:
\[
{\rm Ind}_{\rm hom}(\zeta,0;V) = \sum_{i=0}^n (-1)^{i}
h_i(\Omega^\bullet_{V,0},v)\,,
\]
where $h_i(\Omega^\bullet_{V,0},v)$ is the dimension of the
corresponding homology group as a vector space over $\C$.
\end{definition}

One has:

\begin{theorem}\label{Th indices GSV-hom}
If the germ $(V,p)$ is a complete intersection, then the $\mathrm{GSV}$ and the homological indices of every holomorphic local vector field on $V$ coincide.
\end{theorem}

This theorem was proved by Gómez-Mont in \cite{Gom} when $(V,p)$ is a hypersurface germ, and the general case is proved in \cite{BEG}.  For simplicity 
we  restrict the discussion to the case where $V$ is a
hypersurface. In this case the proof of Theorem \ref{Th indices GSV-hom} in \cite{Gom}  has three main steps:

{\bf Step 1:} Finding an algebraic formula for the  homological index.

{\bf Step 2:}  Observe that the two indices coincide for one vector field if and only if they coincide for every vector filed. 

{\bf Step 3:}  Use the algebraic formula for the homological index and the basic properties of the GSV index, 
to prove  that on every hypersurface germ there exists a tangent holomorphic vector field for which the two indices are $0$.

Let us explain first the algebraic formulas for the homological index.  
Let $\B= \B^{n+1}$ be an open ball around the origin in $\C^{n+1}$. If 
 $f: \B^{n+1} \to \C$ is a holomorphic function with $0 \in
\C^{n+1}$ as its only critical point, then the 1-form
$\sum_{i=1}^{n+1} \frac{\partial f}{dz_i} dz_i$ vanishes only at
$0 \in \C^{n+1}$. Let ${\mathcal
I}_f$ be the Jacobian ideal $(\frac{\partial
f}{dz_1},\cdots,\frac{\partial f}{dz_{n+1}}) \subset \O_{{\C^{n+1}},0}$ of
$f$. Given a vector field $\zeta$ tangent to $V$, with a unique
singularity at $0$, restriction of a vector field $\widehat v =
(a_1,\cdots,a_{n+1})$ on $\B$ with a unique singularity at the
origin, it is shown in \cite[Theorem 2] {Gom} that  the homology
groups
 $H_i(\Omega^\bullet_{V,0},\zeta)$ of the above complex  have dimensions:
\begin{align*}
h_0(\Omega^\bullet_{V,0},\zeta) &=\dim_\C \O_{{\C^{n+1}},0}/(f,a_1,\cdots,a_{n+1}),\\
h_n(\Omega^\bullet_{V,0},\zeta) &= \dim_\C \O_{{\C^{n+1}},0}/(f, {\mathcal I}_f),\\
h_i(\Omega^\bullet_{V,0},\zeta) &=\lambda, \quad \text{for}  \quad
i=1,\dots,n-1,
\end{align*}
where $0 \le \lambda \le \dim_\C \O_{{\C^{n+1}},0}/(f, {\mathcal I}_f)\;$
is the integer:
\begin{align*}
\lambda=& \dim_\C \O_{{\C^{n+1}},0}/(f,a_1,\cdots,a_{n+1})+
\dim_\C \O_{{\C^{n+1}},0}/(\frac{df}{f}(\widehat v),a_1,\cdots,a_{n+1})\\
&-\dim_\C \O_{{\C^{n+1}},0}/(a_1,\cdots,a_{n+1}),
\end{align*}
noticing that the tangency condition for $\widehat v$ means that
$\widehat v (df) $ is a multiple of  $f$, so that $\frac{df}{f}
(\widehat v) $ is a holomorphic function on $\B$.  As a
consequence of these computations G\'omez-Mont  deduced the following expressions for the homological index
(Theorem 1 in \cite {Gom}).

\begin{theorem}\index{Homological index ! algebraic formula}\label{7.2.1}
Let $(V,0)$ be an isolated hypersurface singularity (germ) of
dimension $n$ in $\C^{n+1}$, and let $\zeta$ be the restriction to $V$
of a holomorphic vector field $\widehat v$ on a neighborhood $\B$
of $0$ in $\C^{n+1}$, which has an isolated singularity at $0$ and
is tangent to $V$.

\vv

\noindent  
{\rm (1)} If $n$ is odd, then the homological index of
$\zeta$ is:
\[
{\rm Ind}_{\rm hom}(\zeta,0;V)= \dim_\C \O_{\C^{n+1},0}/(f,
a_1,\cdots,a_{n+1})-\dim_\C\O_{\C^{n+1},0}/(f, {\mathcal I}_f).
 \]

\vv

\noindent
{\rm (2)}  If $n$ is even, then :
\begin{align*}
{\rm Ind}_{\rm hom}(\zeta,0;V)=& \dim_\C
\O_{{\C^{n+1}},0}/(a_1,\cdots,a_{n+1})\\&- \dim_\C
\O_{{\C^{n+1}},0}/(\frac{df}{f}(\widehat v),
a_1,\cdots,a_{n+1})
+\dim_\C \O_{{\C^{n+1}},0}/(f, {\mathcal I}_f).
\end{align*}
\end{theorem}

Notice that in the two formulas above, for even and odd dimensions, the following invariant appears:
$$\dim_\C \O_{{\C^{n+1}},0}/(f,
a_1,\cdots,a_{n+1}) \,.$$
This is a reminiscent of the Tjurina number of holomorphic functions. The following name was  coined in \cite[p. 159]{CaCoMo}:

\begin{definition}\index{Tjurina number ! of a vector field}\label{Tjurina number of a vector field}
Let $\zeta = (a_1,\cdots,a_{n+1})$ be a holomorphic vector field in $\C^{n+1}$ be a holomorphic vector field in $\C^{n+1}$ with an isolated zero at the origin. Then its {\bf Tjurina (-Gómez-Mont)} number at $0$ is:
$$ \tau(\zeta) \, = \dim_\C \O_{{\C^{n+1}},0}/(f,
a_1,\cdots,a_{n+1}) \,.
$$
\end{definition}

The next step in \cite{Gom} for identifying the homological and the GSV  indices for holomorphic vector fields on hypersurface germs is Step 2 above:

 \begin{lemma} The difference
\[\,{\rm Ind}_{\rm hom}(\zeta,0;V) \,-\,  {\rm Ind}_{\rm GSV}(\zeta,0;V) \, = \, k \;,\]
is a constant  that depends only on $V$ and not on the choice of
vector field.
\end{lemma}

The proof of this lemma is an exercise using that both indices satisfy a certain
``Law of conservation of the number’’. This says roughly that if we perturb slightly a holomorphic vector field $\zeta$ at $(V,p)$,  the new vector field $\xi$ still has a singularity at $p$ and new singularities arise at smooth points of $V$, then the index of $\zeta$ at $p$ is the index of $\xi$ at $p$ plus the sum of the PH-indices of singularities of $\xi$ at smooth points of $V$.

The final step for proving Theorem \ref{Th indices GSV-hom} consists in exhibiting a holomorphic vector on every hypersurface germ $(V,p)$ for which 
 the
GSV and homological indices are both $0$. When the dimension  $n$ of $V$ is even, such an example in \cite {Gom} is  a Hamiltonian vector field as in Example \ref{hamilton}. The case when   $n$ is odd is slightly more elaborate. In the two cases the vector field in question  extends in an obvious way to the Milnor fibers with no singularities, and therefore the GSV index vanishes. Then Gómez-Mont uses Theorem \ref{7.2.1} to prove that the homological index vanishes as well.

As a consequence we may state:

\begin{corollary}
Let $(V,0)$ be a hypersurface singularity in $\C^{n+1}$ with an isolated singularity, and let 
$ \zeta = (a_1,...,a_{n+1})$ be a holomorphic vector field  tangent to $V$ with an isolated singularity at $0$. Then the homological index of $\zeta$ at $0$, relative to $V$,  equals the $\mathrm{GSV}$ index and 
 if $n$ is odd, this  is the difference of their Tjurina numbers :
$$ GSV(\zeta,0,V) =    \tau(\zeta) - \tau(V) \,.$$
For $n$ even this is: 
$$ GSV(\zeta,0,V) =   PH(\zeta,0)+\tau(V)-  \dim_\C \frac{\O_{{\C^{n+1}},0}} {\big(\frac{\zeta(f)}{f},
a_1,\dots,a_{n+1}\big)}  \,.$$ 
\end{corollary}

 Let us finish this section with two open questions:
 
\vskip.2cm 

\noindent{\bf Question 1:}
We know that the GSV index can be negative. We also know from \cite{B-GM} that given any isolated hypersurface germ $(V,0)$ there are infinitely many holomorphic vector fields on $V$ with an isolated singularity at $0$, and those with the smallest homological (or GSV) index are dense. What is this index? 

For example, if the singularity is quasi-homogeneous,  this bound seems to be given by the index of the radial vector field, which equals the Euler characteristic of the Milnor fiber. 
\vskip.2cm 
\noindent{\bf Question 2:}
 Let $X$ be a compact complex analytic variety with isolated singularities which are ICIS. Let $\zeta$ be a continuous vector field on $X$ with isolated singularities at the singular points of $X$ and perhaps at some other smooth points of $X$. Assume further that at each point is $X_{sing}$, $\zeta$ is holomorphic, and define its total homological index in the obvious way. Then it is an exercise to show that this is an invariant of $X$, independent of the vector field. What is this invariant? For instance, if $X$ is a local complete intersection, this is the 0-degree Fulton class of $X$, by \cite{SS}. 


\subsection{The virtual index}
The GSV index  inspired also a line of research in the vein of the Baum-Bott’s residues, using Chern-Weil theory:  this is the virtual index,  defined by D. Lehmann, M. Soares and T. Suwa
in \cite{LSS} for vector fields on local complete intersections in  complex manifolds. This was originally defined  for holomorphic  vector fields 
and it was later extended in \cite{SS} for continuous vector fields, which is useful for the study of   characteristic classes on singular varieties (see Section \ref{SS classes}). The virtual index is defined even if the vector field and the variety have non-isolated singularities; it attaches an index in $\Z$ to each compact connected component of the singular set of the vector field. When the singularity is isolated, this coincides with the GSV index.

The starting point for defining the virtual index is recalling that for a continuous vector field $v$ on a complex $m$-manifold $X^m$ with an isolated singularity at a point, say $0$, the local Poincaré-Hopf index can be regarded as being   the localization at $0$ of the top Chern class of $X$. This is explained in detail in \cite{BSS}  using Chern-Weil theory. In the isolated singularity case, which is the most relevant for this article, there is also the article 
  \cite {Suwa2014} by T. Suwa, where he 
  describes  the localization in a particularly nice way, which is relevant for Section  \ref{Poincare Problem}, where  this is recalled (in \ref
{ss. Suwa}). 

Here we give an elementary intuitive explanation of the definition of the virtual index from the topological point of view. Let us explain. 

The tangent bundle $TX$ has Chern classes $c_1(X), ..., c_m(X)$ in the cohomology ring $H^*(X; \Z)$. If $X$ is compact, cap product with the fundamental cycle $[X]$ maps these into homology classes in $H_*(X)$; these are called the homology Chern classes of $X$.  
The Gauss-Bonnet theorem tell us that the image in  $H_0(X) \cong \Z$ of the top Chern class $c_m(X)$ is the Euler characteristic
$\chi(X)$, which equals the total index  of a vector field on $X$ with isolated singularities, by the theorem of Poincaré-Hopf.


More generally, let $v$ be a continuous vector field on $X$ and assume $v$ vanishes on some analytic subset $S$ of $X$. We assume further that there is  a regular neighborhood $\mathcal T$ of $S$ in $X$ such that 
 $v$ is non-singular on  $\mathcal T \setminus S$. Then $V$   splits the Chern class $c_m(X)$ in two parts: one has support  in the interior of $\mathcal T$ and another in its exterior
(see \cite[Chapter 1]{BSS} for details). By Alexander duality, these can be regarded as homology classes in $H_0(\mathcal T)$ and 
$H_0(X \setminus \mathcal T)$.  

\begin{definition}\index{Localization}
The contribution to $c_m(X)[X] = \chi(X)$ coming from $H_0(\mathcal T) \cong H_0(S) $ is {\it the localization of $c_m(X)[X]$} at $S$. 
\end{definition}

If $S$ is a point $p$, then the localization of $c_m(X)[X]$ at $p$ is the corresponding local PH-index. 
 We can always perturb $v$ slightly so that it has isolated singularities in $\mathcal T$ and  the localization of $c_m(X)[X]$ at $S$ is  the total  PH-index at $S$, see \cite[Chapter 1]{BSS}.

We now want to adapt these ideas to vector fields on singular spaces. Let $V$ be a subvariety of the complex manifold $X^m$ defined by a regular section $s$ of a holomorphic vector bundle $E$ over $X$ of rank $k$, with $1 \le k < m$. 

If $s$ is transversal everywhere to the zero-section of $E$, then $V$ is non-singular, the bundle $E|_V$ is isomorphic to the normal bundle of $V$  and the Chern classes  of $V$ can be regarded as being those of the virtual bundle $TX\vert_V - E\vert_V$. Otherwise  $V$ is a singular  local complete intersection and one defines its  {\it virtual tangent bundle}\index{Virtual tangent bundle} by.
$$ \tau V:= TX\vert_V - E\vert_V \;.$$
 The total Chern class  of  $ \tau V$ is  defined in the usual way:
$$c_*(\tau V) = c_*(TX\vert_V ) \cdot c_*(E \vert_V ) ^{-1} \,,$$
where  the term in the extreme right is the inverse of $ c_*(E \vert_V ) $ in the cohomology ring. 

One gets the {\it total Fulton class} $c_*^{Fu}(V) = c_0^{Fu}(V) + .... + c_n^{Fu}(V)$~\index{Fulton classs} by taking the cap product of $c_*(\tau V) $ with the fundamental cycle $[V]$. Fulton in \cite{Fu} calls this the canonical class. 
The components $c_i^{Fu}(V) \in H_{2i}(V)$, for   $i= 0,...,n$,  are 
the {\it Fulton classes} of $V$. 

We remark that in \cite {Fu} Fulton actually defines these as classes of cycles in the Chow group and calls $c_*^{Fu}(V)$ the canonical class of $V$. His definition uses the Segre classes and it works for singular varieties in full generality (we refer to \cite {BSS, Aluffi6, Bra, CMS-2}  for thorough accounts, with different viewpoints,  on Chern and their relations with singularity theory).

The idea now is to localize the zero-degree Fulton  class (which corresponds to the top dimensional class in cohomology) using a tangent vector field, as we did on manifolds. The problem is that    $\tau V$ is a virtual bundle and the usual methods for localizing Chern classes via topology do not work. 

Localization is done in   \cite {LSS, BLSS2, BSS} by means of Chern-Weil theory and  it works equally well if we replace the isolated singularity $p$ by a compact connected component $S$ of the singular set. We refer to \cite{BSS} for details. 

Let us explain this construction  in a special case which is particularly interesting and where it is much simpler. Assume that the regular section $s$ that defines $V$ can be put in a flat family $\{s_t\}_{t\in \C}$ of holomorphic sections of $E$ such that $s_0 = s$ and for $t\ne 0$ the section $s_t$ is transversal to the zero section. 
Set $V_t := \{ x \in X \, | \, s_t(x) = 0 \}$. 

We have a flat family of manifolds $V_t$ degenerating to the special fiber $V$. Over each $V_t$ one has the virtual bundle 
$$TX|_{V_t} - E|_{V_t} \,.$$ 
For  $t \ne 0$ the bundle $E|_{V_t}$ is isomorphic to the normal bundle of $V_t$, so the Chern classes of $TX|_{V_t} - E|_{V_t}$ are the Chern classes of  the tangent bundle $TV_t$, and these can be regarded in homology. Denote these by $c_{*,t} \in H_*(V_t)$.

Now let $\widetilde {\mathcal T}$ be a regular neighborhood of $V$ in $X$, and  let $\e_o>0$ be small enough so that whenever $|t| < \e_o$,  the manifold $V_t$ is contained in $\widetilde {\mathcal T}$. This exists 
by the compactness of $V$.  Since  $\widetilde {\mathcal T}$ has $V$ as a deformation retract, one has well defined homomorphisms:
$$ H_*(V_t) \buildrel{\iota_*} \over {\to}   H_*(\widetilde {\mathcal T})     \buildrel{r_*} \over {\to}   H_*(V) \,,  $$
 induced by the inclusion of $V_t$ in $\widetilde {\mathcal T}$ and  a retraction to $V$. Set 
$\mathcal S =  r_* \circ \iota_*$. This is a special case of Verdier’s specialization \cite{Verdier2}, and we know that the image $\mathcal S(c_{*,t}) \in H_*(V)$ are the Fulton classes of $V$.
 Let us summarize this discussion:
 
 \begin{proposition}
 The Verdier specialization morphism $\mathcal S$ carries the homology Chern classes of the general fiber $V_t$ into the Fulton classes of the special fiber $V$. 
\end{proposition}

We now explain  how to localize the top Fulton class of $V$ at each connected component of its singular set $V_{Sing}$  using a vector field. Intuitively this is easy: we use the vector field as we did on manifolds, to localize a contribution to the index at a component of the singular set of the vector field, and carry this contribution to the special fiber using the specialization morphism.

For simplicity we do it when the variety $V$, defined by a regular section of a holomorphic bundle $E$ over $X$, only has isolated singularities, say $p_1,..., p_r$. The process in general is similar but a little more difficult technically. 

For each $p_i$ we may consider a conical neighborhood $X_i$ of $p_i$ in $V$ and “replace it” by a local Milnor fiber $F_{i,t}$. We get a compact manifold $\widehat V_t$, and as $|t|$ tends to $0$ this degenerates to $V$.  As above, the homology Chern classes of $V_t$ are mapped to the Fulton classes of $V$. Now, given a vector field on  $V$ with isolated singularities, we can take its restriction to the compact manifold with boundary
$V \setminus (\bigcup_{i= 1}^r X_i\bigcup)$, and extend it to all of $\widehat V_t$ with isolated singularities on each local Milnor fiber $F_{i,t}$. This splits $c_n(\widehat V_t)[V_t]$ into a contribution coming from the smooth part $V \setminus (\bigcup_{i= 1}^r X_i \bigcup)$, and a contribution coming from each $F_{i,t}$.

\begin{figure}
	\centering
\quad \includegraphics[height=5cm ]{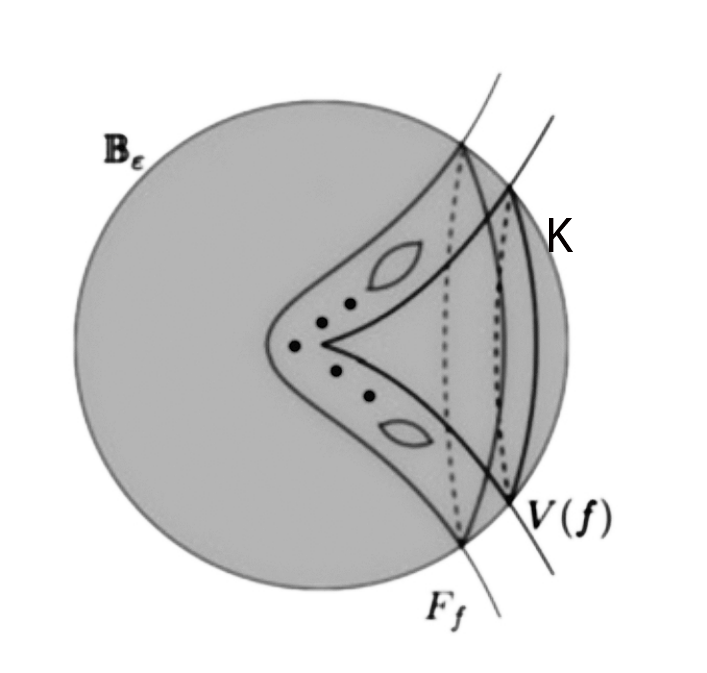}
	\caption{The  virtual index is the localization at $0$ of the PH index  in a Milnor fiber.}
\end{figure}

\begin{definition}\index{Virtual index}
The virtual index of $v$ at an isolated singularity $p$ is the localization at $p$ of the Fulton class $c_n(V) \in H_0(V) \cong \Z$ determined by $v$. 
\end{definition}

From the above construction we see that this invariant equals the total Poincaré-Hopf index of $v$ at a local Milnor fiber when the singularity $(V,p)$ can be put in a flat family as above. This condition is not necessary and the result id proved in general, either in \cite{BSS, BLSS2} via Chern-Weil theory, or in \cite {CMS-2} topologically. We have:

\begin{theorem}\label{Th GSV vs virtual}
The {\rm GSV} and the virtual indices coincide for vector fields on ICIS germs.
\end{theorem}

This theorem is discussed again in Section  \ref{Poincare Problem} from the viewpoint of Chern-Weil theory, following Suwa’s viewpoint in   \cite {Suwa2014} instead of the original one in \cite{LSS} (or in \cite{BSS}).

\vskip1cm

\subsection {Radial index and the  Milnor number}

There are  other  indices of vector fields and 1-forms on singular varieties that play an important role in singularity theory. Two of these are 
 the radial or Schwartz index, and the local Euler obstruction, which  plays an important role for MacPherson’s construction of Chern classes for singular varieties. This invariant has been extended to an index for vector fields and 1-forms on ICIS singularities.  We refer to  \cite{BSS, EG, CMS-2} for thorough accounts on this subject.

 The radial (or Schwartz) index\index{Radial index}  springs from the work of M. H. Schwartz in 1965, but she defined it only for a special class of vector fields, those obtained by radial extension. Her goal was  defining Chern classes for singular varieties (see \cite{BSS}).  For instance, if $p$ is an isolated singularity in an analytic space $X$, the only vector fields considered by M. H. Schwartz were the radial ones, {\it i.e.,} those transversal to each small sphere in the ambient space, centered at the singular point, and in that case the Schwartz index is 1. 
 King and Trotman  defined an index in general in \cite{KT}, a 1991 preprint that unfortunately was not  published till 2014, which essentially is the radial index. In the meantime, the same index was discovered simultaneously in \cite{EG1} and \cite{ASV}. 
 This is defined in great generality, for real or complex analytic spaces,  and we refer to  \cite{EG} and  \cite[Ch. 2]{BSS} for accounts on the subject. In the simplest case, where one considers an isolated singularity germ $(V,0)$ and a continuous vector field on $V$ with an isolated singularity at $0$, this index is defined as: 
$$ \mathrm{Rad}(v; 0,V) = 1 + d(v,v_{\rm rad}) \,,$$
where $v_{\rm rad} $ is a radial vector field on $V$ ({\it i.e}, it is transversal to the intersection of $V$ with every sufficiently small sphere, pointing outwards) and  $d(v,v_{\rm rad})$ is {\it the difference between these two vector fields}, as depicted in   figure  \ref{figure radial}:   this is the total index of a vector field $\xi$ on a cylinder bounded by links $\mK_{\e’} = V \cap \s_{\e’} $  and $\mK_{\e} = V \cap \s_{\e}$ with $0 < \e’ < \e$ and $\xi$ restricts to $v_{\rm rad}$ on $\mK_{\e’}$ and to $v$ on $\mK_{\e}$.

\begin{figure}[h!]
	\centering
\quad \includegraphics[height=5cm ]{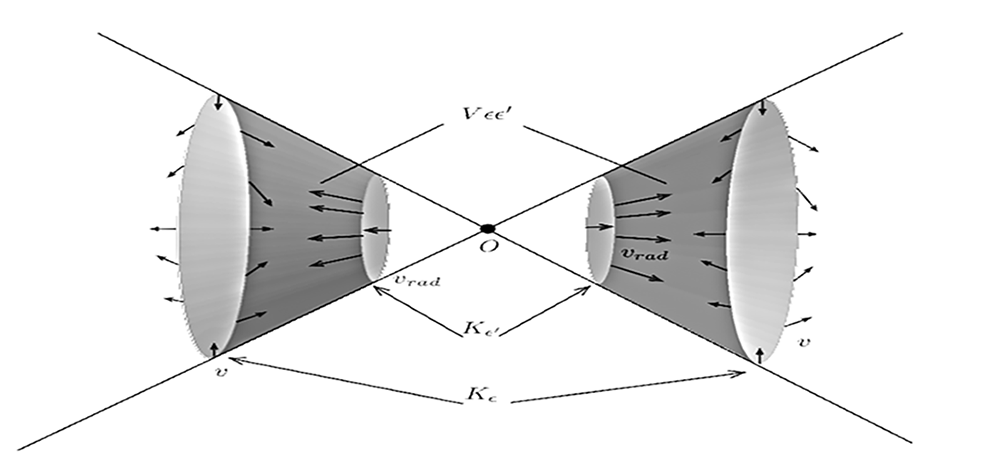}
	\caption{The  radial index}
    \label{figure radial}
\end{figure}

When $V$ has non-isolated singularities, the radial index is defined similarly as above but  there are two points of view.  In \cite[Ch. 2]{BSS} one assigns an index to each connected component $\Sigma$ of the singular set of a singular variety. For this one takes a regular neighborhood $N_\Sigma$ of $\Sigma$ with smooth boundary, and a smaller such neighborhood $N_\Sigma’$ and use the cylinder bounded by these to define the difference  $d(v,v_{\rm rad})$, where $v$ is the given vector field and $v_{\rm rad}$ is now the unit normal vector field on the boundary $\partial N_\Sigma’$, pointing outwards. In this case we define the radial index of  $v_{\rm rad}$  to be the Euler characteristic $\chi(\Sigma)$ and the index of $v$ is $\chi(S) + d(v,v_{\rm rad})$. 

In \cite{EG1, EG} the definition is more refined as one must consider  Whitney stratified vector fields with isolated singularities and one defines a radial index at each singularity of the vector field. 
The relation between the two definitions is as follows: the above radial index associated to $\Sigma$ is the sum of the radial indices of the stratified vector field with isolated singularities contained in $\Sigma$. In either case the radial index at $\Sigma$ actually is a class in $H_0(\Sigma; \mathbb Z) \cong  \mathbb Z$, it can be regarded in $H_0(V; \mathbb Z)$ via the inclusion,  
and it has the following important properties (see \cite{BSS, EG}):

\begin{theorem}
\begin{enumerate}
\item Let $(V,0)$ be an $n$-dimensional ICIS  with  Milnor number $\mu$, and 
 let $v$ be a continuous vector field on $V$ with an isolated singularity at $0$. Then their difference is the Milnor number of $(V,0)$, independently of $v$: $$ \mu = (-1)^n \big( \mathrm{GSV}(v;0,V)  -  \mathrm{Rad}(v;0,V)  \big) \,.
$$

\item  If $V$ is a compact local complete intersection and $v$ is a continuous vector field on $V$, singular at the singular set $V_{sing}$ and at finitely many smooth points of $V$, then its total radial index (defined in the obvious way) is the Euler characteristic $\chi (V)$, which equals the $0$-degree Schwartz-MacPherson class of $V$. 
\end{enumerate}
\end{theorem}

\begin{remark}[Indices of 1-forms]\index{Index of a 1-form}
The theory of indices of 1-forms on singular varieties is parallel to that for vector fields and it has been developed mostly by W. Ebeling and S. Gusein-Zade in a series of papers. We refer to their survey paper \cite{EG} for an account of this, and to \cite[Ch. 9]{BSS} for an introduction to the subject. 
Given an isolated singularity  germ $(V,p)$ in some $\C^m$ and a holomorphic 1-form in a neighborhood of $p$, its restriction to $V$ is a 1-form on $V$ and its singularities are the points where its kernel is not  transversal to $V$; this includes $p$. There is a radial index for 1-forms with an isolated singularity at $p$, and if $(V,p)$  is an ICIS, then there are also GSV and homological indices, and these coincide. The difference between these and the radial index is the local Milnor number.  There is  a local Euler obstruction as well. 

Using 1-forms instead of vector fields has some advantages; perhaps a main one is that instead of imposing the tangency conditions one needs for vector fields, one needs a transversality condition for 1-forms. This allows, for instance, considering indices for collections of 1-forms. The indices we have envisaged so far morally correspond to the top Chern class, while the collections of 1-forms envisaged by Ebeling and Gusein-Zade corespond to different Chern numbers.
\end{remark}

\begin{remark}[Milnor classes]\index{Milnor class}\label{SS classes}
Classical Chern classes of complex manifolds can be defined as being the primary obstructions to constructing frames with appropriate indices at their singularities. These indices are natural extensions of the Poincaré-Hopf index. Similarly, for singular varieties equipped with a Whitney stratification, one may consider appropriate stratified frames and:
\begin{itemize}
\item One has a corresponding radial index for frames; the classes one gets are the Schwartz-MacPherson classes (see \cite[Ch. 10]{BSS}).
\item If the variety is a local complete intersection in a manifold, one has a corresponding virtual index for frames,  and the classes one gets are the Fulton classes (see \cite[Ch. 10]{BSS}).
\item In the previous setting, the difference between the Fulton and the Schwartz-MacPherson classes are the Milnor classes of $V$. 
\item If $V$ has only isolated singularities, the Milnor class in degree $0$ is the sum of the local Milnor numbers (hence the name of Milnor classes). If $V$ is a hypersurface with non-isolated singular set, the 0-degree Milnor class  is Parusinski’s generalized Milnor number. 
\end{itemize}
Milnor classes arose first in the work of P. Aluffi \cite{Aluffi2, Aluffi1}, though the actual name was coined later (see \cite {CMS-2}).  
\end{remark}

\section{Singular Holomorphic foliations, Baum-Bott  residues and indices }
A  \textit{regular} (smooth or non-singular)  holomorphic foliation $\F$ of dimension $k$ on a complex manifold $X$~\index{Holomorphic foliation !} 
  is given by an open covering $\{U_{\alpha}\}$  of $X$, 
holomorphic  submersions \\ $ f_{\alpha}:U_\alpha \to \mathbb{C}^{n-k}$  and  biholomorphisms 
$$g_{\alpha \beta}:f_\beta(U_{\alpha}\cap U_{\beta})\subset \mathbb{C}^{n-k} \to f_\alpha(U_{\alpha}\cap U_{\beta})\subset \mathbb{C}^{n-k}$$
 such that   $f_\alpha = g_{\alpha \beta}\circ f_\beta$, whenever \( U_{\alpha} \cap U_{\beta} \neq \emptyset \).
 That is (see figure \ref{foliatedfig}), the diagramm 
$$
\xymatrix@R=4pc{ 
U_{\alpha} \cap U_{\beta} \ar[d]_{f_\beta} \ar@{=}[r] & U_{\alpha} \cap U_{\beta} \ar[d]^{f_\alpha} \\
f_\beta(U_{\alpha} \cap U_{\beta}) \subset \mathbb{C}^{n-k} \ar[r]^{g_{\alpha \beta}} & f_\alpha(U_{\alpha} \cap U_{\beta}) \subset \mathbb{C}^{n-k}
}
$$
commutes.

\begin{figure}[h!] 
    \centering
    \includegraphics[width=0.4\textwidth]{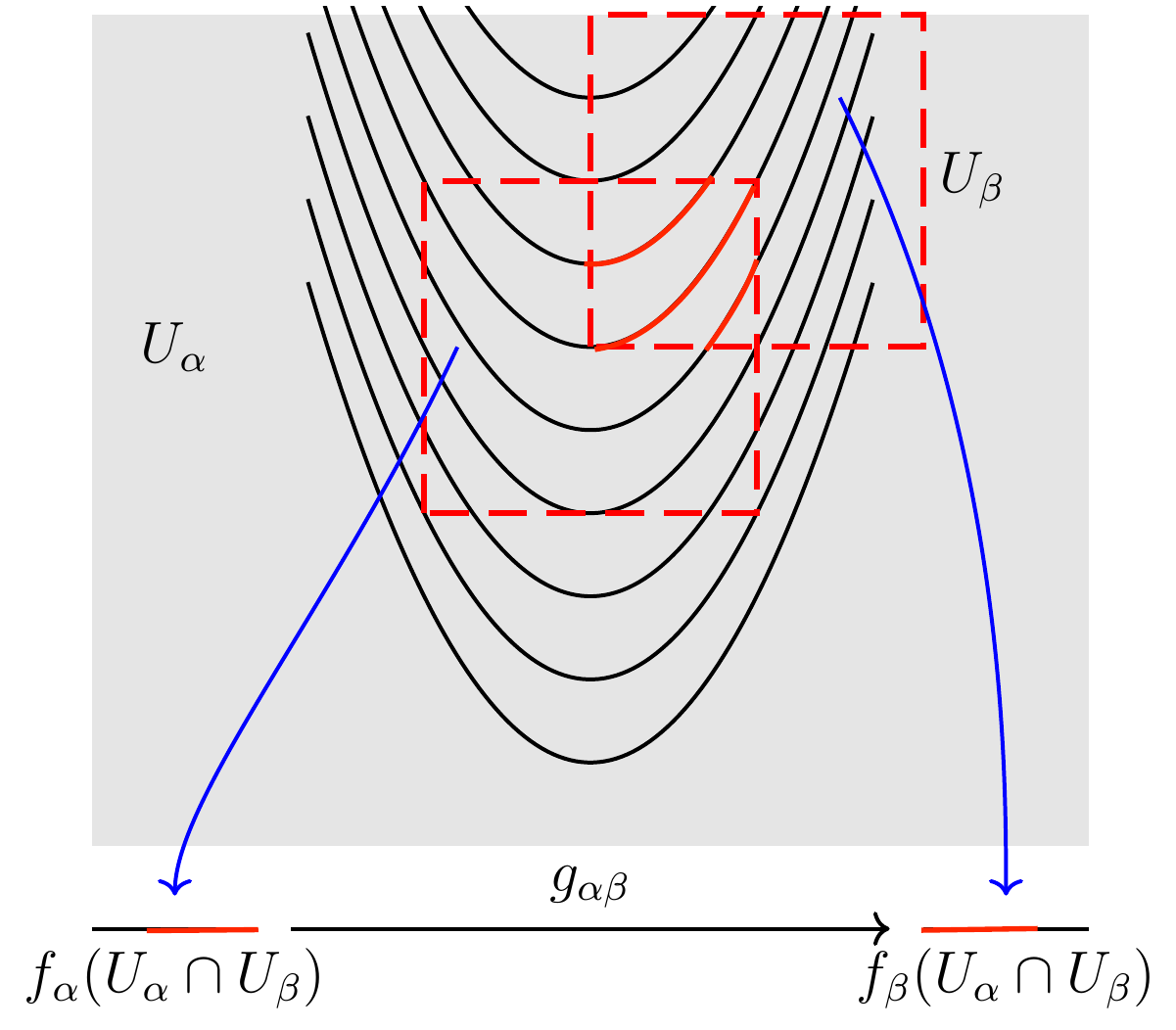}
    \caption{}
    \label{foliatedfig}
\end{figure}
If we set $f_{\alpha}=(f_{\alpha}^1,\dots,f_{\alpha}^{n-k})$, then the tangent of a fiber passing through $p\in U_\alpha$ is: $$T_p(f_{\alpha}^{-1}(f_{\alpha}(p)))=\ker(df_{\alpha})(p)=\{v\in T_pX; i_v(df_{\alpha}^1\wedge \cdots \wedge  df_{\alpha}^{n-k})(p)=0\}.$$ 
This means that the foliation is locally induced by the decomposable  $(n-k)$-form
$$df_{\alpha}^1\wedge \cdots \wedge  df_{\alpha}^{n-k}\in H^0(U_\alpha,\Omega_{U_\alpha}^{n-k}) \,,$$
and on $U_\alpha\cap U_\beta\neq \emptyset$ we have
$$
df_{\alpha}^1\wedge \cdots \wedge  df_{\alpha}^{n-k}= \det(J(g_{\alpha \beta}))  \cdot df_{\beta}^1\wedge \cdots \wedge  df_{\beta}^{n-k},
$$
where $J(g_{\alpha \beta})$ is the jacobian matrix  of the biholomorphic map  $g_{\alpha \beta}$. 
Moreover, 
this gives us a holomorphic vector bundle of rank $k$ defined by $$T\F:=\cup_{\alpha} \ker(df_{\alpha})\subset TX $$ called the tangent bundle of $\F$. On the other hand,
$T\F|_{U_{\alpha}}$ is generated by holomorphic vector fields $v_{\alpha}^1,\cdots,v_{\alpha}^k$, such that $v_{\alpha}^i=\sum_{j=1}^k\Gamma^k_{ij}v_{\alpha}^j$, where $\Gamma^k_{ij}$ are holomorphic functions on $U_{\alpha}$. That is, the foliation is induced by the decomposable holomorphic  $k$-vector field 
$$v_{\alpha}^1\wedge \cdots \wedge v_{\alpha}^k \in H^0(U_\alpha,\wedge^{k}TU_\alpha).$$ 
It follows from the Frobenius theorem that a holomorphic subvector bundle $\mathcal{E}\subset TX$ (called a distribution) of rank $k$   is the tangent bundle of a foliation of dimension $k$ if and only if $[\mathcal{E},\mathcal{E}]\subset \mathcal{E}$. 
So, a holomorphic regular foliation induces an exact sequence of holomorphic vector bundles, 
\begin{equation*}
  0  \longrightarrow T\F \stackrel{\phi}{ \longrightarrow} TX \stackrel{\pi}{ \longrightarrow} N\F  \longrightarrow 0,
\end{equation*}
where the quotient bundle $N\F$ is called the normal bundle of $\F$.  Notice that $\{\det(J(g_{\alpha \beta}))\}$ is a cocycle of $\wedge^{n-k}N\F$ the determinant of $\F$. Then, the foliation $\F$ induces a   nowhere vanishing holomorphic twisted  $(n-k)$-form, $$\omega_{\F}\in H^0(X, \Omega^{n-k}_X\otimes \wedge^{n-k}N\F ) \,,$$ with values in the line bundle $\wedge^{n-k}N\F$, given by $\omega_{\F}|_{U_\alpha}=df_{\alpha}^1\wedge \cdots \wedge  df_{\alpha}^{n-k}$.  Reciprocally,  a      nowhere vanishing twisted holomorphic  $(n-k)$-form $\omega\in H^0(X, \Omega^{n-k}_X\otimes \L )$, with values in  a line bundle $\L$ on $X$, induces a regular holomorphic foliation of dimension $k$, which is locally decomposable  and integrable.  That is, for all $p$ there is an open  neigborhood $U$ and $\omega_1,\dots,\omega_{n-k}\in H^0(U,\Omega^{n-k}_U)$, such that 
$$
\omega|_{U}=\omega_1 \wedge \cdots\wedge \omega_{n-k}
$$
and 
$$
d\omega_i \wedge \omega_1 \wedge  \cdots\wedge \omega_{n-k}=0
$$
for all $i=1,\dots,n-k$. Then, by the Frobenius Theorem, there is a holomorphic map $f=(f^1,\dots,f^{n-k}):U\to \mathbb{C}^{n-k}$, such that 
$$
\mbox{ker}(\omega|_{U})=\{v\in TU; i_v(\omega)=0\}=\mbox{ker}(df)=\mbox{ker}(df^1\wedge \cdots \wedge  df^{n-k})\subset TU. 
$$
Therefore, $\omega$ induces a holomorphic foliation $\F$ of dimension $k$, with  tangent bundle $T\F=\mbox{ker}(\omega)\subset TX$. Moreover,  for  every $g\in H^0(X,\mathcal{O}_X^*)$ the form $\Tilde{\omega}=g\omega$, yields the same foliation. In particular, if $X$ is compact, then the  foliation $\F$ is represented by a point on the projective space
$$
\mathbb{P} H^0(X, \Omega^{n-k}_X\otimes \L ). 
$$
In general,  we can define.

\begin{definition}\index{Pfaff system ! regular}
    A \textit{regular Pfaff system} of dimension $k$ on  a \textit{compact} manifold $X$ is  a point $[\omega]\in \mathbb{P} H^0(X, \Omega^{n-k}_X\otimes \L )$,  for some line bundle  $\L$,  such that its singular set is $\mbox{Sing}(\omega):=\{\omega=0\}=\emptyset$. 
\end{definition}
So, a regular Pfaff system $\omega$ is a regular foliation if and only if  $\omega$ is \textit{locally decomposable} and \textit{integrable}.


\subsection{Singular holomorphic foliations and Pfaffian systems}
Let $X$ be a complex manifold of dimension $n$. A singular holomorphic   \emph{distribution} $\sF$, of dimension $k$,  on $X$ is given by an exact sequence
\begin{equation*}
\F:\  0  \longrightarrow T\F \stackrel{\phi}{ \longrightarrow} TX \stackrel{\pi}{ \longrightarrow} N{\F}  \longrightarrow 0,
\end{equation*}
where $N\F$ is a torsion-free sheaf, called the \emph{normal sheaf} of $\sF$, and $T\F$ is a reflexive sheaf, of generic rank $k$, called the \emph{tangent sheaf} of $\sF$.

\begin{definition}\index{Holomorphic foliation ! singular}
    A \emph{singular holomorphic  foliation} of dimension $k$ is a distribution $\sF$, of dimension $k$, which is involutive, i.e. $[T\F, T\F] \subset T\F$.
The \textit{singular set} of the foliation  $\F$ is the singular set of its normal sheaf
$$
\operatorname{Sing}(\mathcal{F}) := \operatorname{Sing}(N\mathcal{F}) = \{p \in X \mid (N\mathcal{F})_p \text{ is not locally free}\}.
$$
\end{definition}

Since $ N\mathcal{F}$ is torsion-free, the singular set of $\F$ has codimension $\geq 2$.  The \emph{singular scheme}\index{Holomorphic foliation ! singular scheme} of $\F$ is defined as follows. Taking the maximal exterior power of the dual morphism $\phi^\vee:\Omega^1_X\to T\sF^*$ we obtain a morphism 
$$\wedge^k \phi^\vee: \Omega^{k}_X\to \det(T\F)^*
$$   and an induced morphism $\Omega^{k}_X\otimes \det(T\F)\to \mathcal{O}_X$ whose    image   is an ideal sheaf $\mathcal{I}_Z$ of a subscheme $Z\subset X$. Moreover, such a morphism induces a global section $$\nu_{\F}\in H^0(X, \wedge^k TX\otimes \det(T\F)^*)$$
which is called the \textit{Pfaff field} induced  by $\F$. This means that on a neighborhood $U_p$ of a point $p$ such that $T\F$ is locally free and  generated by vector fields $v_1,\dots,v_k$, then $\nu_{\F}|_{U_p}=v_1\wedge\cdots \wedge v_k$ and $[v_i,v_j]=\sum_{j=1}^k\Gamma^{k}_{ij}v_k$, for all $i,j,k=1,\dots,n$, where   $\Gamma^{k}_{ij}\in \mathcal{O}_{U_p}$ and $v_i,v_j\in H^0(U_p, TU_p)$. Therefore, a singular holomorphic foliation $\F$  of dimension $k$, induces a regular foliation $\F^0$ of dimension $k$ on $X^0:=X-\mbox{Sing}(\F)$.   

The line bundle $K\F:=\det(T\F)^*$ is called the \textit{canonical bundle} of $\F$. So, By the isomorphism $\wedge^k TX\simeq \Omega_X^{n-k}\otimes \det(TX)$ and the fact that $\det(N\F)\simeq \det(TX)\otimes K\F$  we conclude that the isomorphism
$$ \wedge^k TX\otimes K\F\simeq \Omega_X^{n-k} \otimes 
\det(TX)\otimes K\F\simeq 
\Omega_X^{n-k} \otimes\det(N\F)$$
says   that the section $\nu_{\F}$ corresponds to a twisted $(n-k)$-form $$\omega_{\F}\in H^0(X, \Omega_X^{n-k}\otimes\det(N\F))$$ with values on $\det(N\F)$ which yields  a \textit{Pfaff system},   of   dimension $(n-k)$,   associated with the corresponding regular foliation  $\F^0$ on the regular part $X^0=X-\mbox{Sing}(\F)$.

So, in a neighborhood $U_p$ of a point $p$ where $N\mathcal{F}^*$ is locally free, we have
\[
\omega_{\mathcal{F}}|_{U_p} = \omega_1 \wedge \cdots \wedge \omega_{n-k},
\]
where $ \omega_i \in H^0(U_p, \Omega^1_{U_p}) $ satisfies the integrability condition
\[
d\omega_j \wedge \omega_1 \wedge \cdots \wedge \omega_{n-k} = 0
\]
for all $ j = 1, \dots, n-k $. Therefore,
$$
\operatorname{Sing}(\mathcal{F})=\{\nu_{\mathcal{F}}=0\}=\{\omega_{\mathcal{F}}=0\}. 
$$
\begin{definition} \index{Pfaff system ! singular}
Let $X$ be a compact complex manifold of dimension $n$. 
    A \textit{Pfaff system}  of dimension $k$ on $X$ is an element $$[\omega] \in \mathbb{P} H^0(X, \Omega_X^{n-k}\otimes \mathcal{L}),$$  where $\mathcal{L}$ is a holomorphic line bundle on $X$, such that its  singular set $\operatorname{Sing}(\omega):=\{\omega=0\}$ has codimension $\geq 2$. 
\end{definition}

\begin{remark} 
Note that a holomorphic foliation of dimension $k$ induces a Pfaff system of codimension $n-k$ that is generically decomposable. However, in general, a Pfaff system does not need to be generically decomposable. An example of a non-decomposable Pfaffian system is given by the symplectic $2$-form  in \( \mathbb{C}^{2m} \). In canonical coordinates \( (q_1, p_1, \dots, q_m, p_m) \in \mathbb{C}^{2m} \), this symplectic form  is expressed as
\[
\omega = \sum_{i=1}^m dq_i \wedge dp_i. 
\]
\end{remark}

\begin{example}
    \label{folPn}
A holomorphic foliation $\sF$ on $\pn$ of dimension $k$ corresponds to a twisted $(n-k)$-form $\omega \in H^0(\pn, \Omega_{\pn}^{n-k}(d+n-k+1)$, where $d\geq 0$ is called the \emph{degree} of $\sF$. We  can interpret $\omega$ as a polynomial differential form
\[
\omega = \sum_{1\leq i_1< \cdots < i_{n-k}\leq n} A_{i_1\cdots i_p}dx_{i_1}\wedge \cdots \wedge dx_{i_q}
\]
where the $ A_{i_1\cdots i_p}$ are homogeneous of degree $d+1$, and  $\iota_{R} \omega = 0$, where  $$\mbox{R} = x_0\frac{\partial}{\partial x_0} + \cdots + x_n\frac{\partial}{\partial x_n}$$ is the radial vector field.  The local decomposability translates to the Pl\"{u}cker conditions
$$
    ( {i}_v\omega) \wedge \omega = 0 \quad \text{ for every }v
 \in \wedge^{n-k-1}(\mathbb C^{n+1}) ,
$$
and the integrability condition becomes
$$
     ( {i}_v\omega) \wedge d\omega = 0 \quad \text{ for every }v
 \in \wedge^{n-k-1}(\mathbb C^{n+1}) .
$$
\end{example}
For more details, we refer the reader to \cite[Section 1.3]{CukiermanPereira2008}.

\begin{example}
Let $\psi: X \to Y$ be a holomorphic fibration, where $Y$ is a complex manifold of dimension $n - k > 0$. Then $\psi$ induces a holomorphic foliation $\mathcal{F}$ on $X$ of dimension $k$, whose tangent sheaf is the relative tangent sheaf of $\psi$, and whose singular set corresponds to the singularities of the fibration.
\end{example}

\begin{definition} \index{Invariant variety}
 Let $\omega \in  H^0(X, \Omega_X^{n-k}\otimes \mathcal{L})$ be a Pfaff system of dimension $k$ on a complex manifold $X$. 
We say that an analytic subvariety $V \subset X$ is  {\it invariant} by $\omega$ if 
$
i^*\omega \equiv 0,
$
where $i: V\hookrightarrow X$ is the inclusion map. If $\mathcal{F}$ is a holomorphic foliation of dimension $k$ on $X$, then $V$ is invariant by $\mathcal{F}$ if it is invariant by its corresponding Pfaff system 
$$
\omega_{\mathcal{F}} \in H^0(X, \Omega_X^{n-k} \otimes \det(N\mathcal{F})).
$$

\end{definition}
\begin{example}
 Consider the foliation $\F$ in  $\mathbb{C}^2$ induced by the holomorphic 1-form $\omega = x(5y^2 - 9) \, dx - y(5x^2 - 9) \, dy.
$ This foliation has 5 singular points  given by the singular  points  of the invariant  algebraic curve $C=\{(x^2-y^2)\cdot (x^2+4y^2-9)=0\}$. See figure \ref{fig:exeflow}.   
\end{example}
\begin{figure}[h!] 
    \centering
    \includegraphics[width=0.4\textwidth]{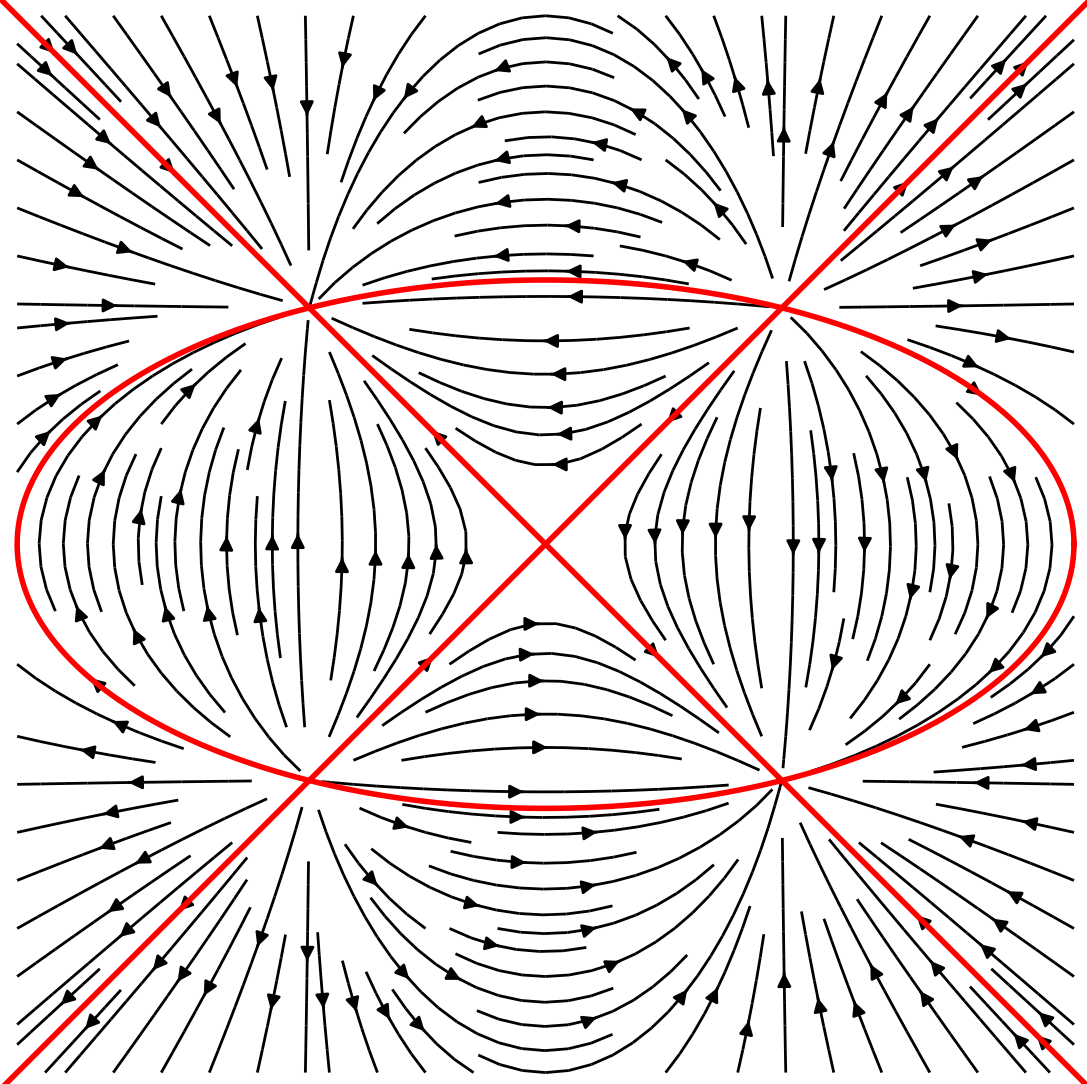}
    \caption{}
    \label{fig:exeflow}
\end{figure}

\begin{example}
In this example, we examine the topology, via the Milnor fibration,  of  a  foliation on $\mathbb{C}^2$ induced by an exact holomorphic 1-form $\omega = df$, i.e., when the foliation admits a holomorphic first integral. Suppose that the foliation is induced by the 1-form  $d(x^3-y^2)$ with an isolated singularity at $0$.
This foliation is tangent to the fibers of  the polynomial  function $f: \mathbb{C}^2\to \mathbb{C} $, given by $f(x,y)=x^3-y^2$, and it has  the invariant (separatrix) curve $f^{-1}(0):=C$ with isolated singularity at $0$.
Pick the  point $0 \in C$ and a small
enough $\epsilon$-ball  $B_\epsilon(0)$ in 
 $\mathbb{C}^2$ centered at the origin  $0$, with boundary the  3-sphere  $S_\epsilon(x)$. Then by Milnor \cite{Mi1},  $B_\epsilon(0)\cap C$  is homeomorphic to the cone on the real link  $S_\epsilon(0)\cap C$. In  this case,
  the link $S_\epsilon(0)\cap C\simeq \mathbb{S}^1$ is the trefoil knot. If we take a point $x_1\in D_\delta(0)-\{0\} \subset \mathbb{C}$, with $|x_1| \ll \delta \ll \epsilon\leq 1 $, then on the smooth curve $f^{-1}(x_1)$ we have also  that $S_\epsilon(0)\cap f^{-1}(x_1)$ is the trefoil knot, while the leaves of the foliation contained within the ball, which are the Milnor fibers, have the homotopy of a bouquet of 2 circles, $\simeq \mathbb{S}^1\vee \mathbb{S}^1$.   See figure \ref{MF}.
\end{example}


\subsection{Baum--Bott residues}
In this section we recall the Baum–Bott residue theorem for singular foliations on complex manifolds.  A starting point is Bott’s vanishing theorem presented by R. Bott in the 1970 ICM in Nice (see \cite{Bott1970}), where he discussed the state of the art in the problem of “foliating a manifold”. As he explains, this can be viewed in two steps:

\nn
{\bf Step 1: }
Given a (smooth or complex) $n$-manifold, construct a $k$-dimensional subbundle $E$ of the tangent bundle $TM$; and then, 

\nn {\bf Step 2:}  Given  a subbundle $E$ of $TM$ as above,  can it be deformed to an integrable one?

In the complex case we want $E$ to be   holomorphic.

Bott then gives a beautiful answer, which in the complex case can be stated as follows:

\begin{theorem} [Bott’s vanishing theorem]\label{Bott theorem}\index{Bott Vanishing Theorem}
Let $M$ be a compact complex analytic $n$-manifold which admits an integrable holomorphic subbundle $E$. Let $Q= TM/E$ be the quotient bundle, that one calls the normal bundle. Then the Chern ring $TM/E$, generated by the real Chern classes of $TM/E$, vanishes in dimensions greater that $2 \times {\rm dim}_\C \,(TM/E)$. That is,
$$Chern^r(TM/E) = 0 \; \quad  \forall \, r > {\rm dim}_\R \,(TM/E)\;.$$
\end{theorem}

This provides topological obstructions for the integrability of a distribution on a manifold. It is worth saying that there is also previous work by Bott (e.g. \cite{Bott1967}) where he  relates the behavior of a holomorphic vector field at its zeroes   to the characteristic numbers of complex manifolds and holomorphic  bundles. 

As an example (from \cite{Bott1970}),  consider a compact complex $n$-manifold $M$ with vanishing Euler characteristic. 
This means that its top Chern class $c_n(M)$ vanishes as well, and also that by the theorem of Poincaré-Hopf, we can construct on $M$ a nowhere vanishing $C^{\infty}$ vector field $v$. This spans a  smooth  trivial 1-dimensional complex sub-bundle $E$ of $TM$ and one has a $C^{\infty}$ ssplitting $TM \cong E \oplus Q$ where $Q = TM/E$ is the normal bundle. Hence, since $E$ is a trivial bundle,  all the Chern numbers of $M$ are degree $n$ products of the Chern classes of $Q$, evaluated on the fundamental cycle $[M]$.

Now suppose that the vector field $V$ actually is holomorphic, so it is integrable. Then one has a 1-dimensional holomorphic foliation $\mathcal E$ with tangent bundle $E$.  Then, by Bott’s theorem, all degree $n$ products of Chern classes of $Q$ vanish. Hence, by the above observations, all the Chern numbers of $M$ vanish, so we have:

\begin{corollary}
If a compact complex manifold admits a nowhere vanishing holomorphic vector field, then all its Chern numbers are $0$. 
\end{corollary}

Bott’s theorem \ref{Bott theorem} was later extended by Baum-Bott to higher dimensional foliations, allowing also singularities. If there are no singularities, then the theorem states that the 
   ring $Chern^r(Q)$, generated by the real Chern classes of the normal bundle, vanishes  in dimensions greater that the real codimension of the foliation. In the presence of singularities, the theorem says that each such class, regarded in homology via Alexander duality, localizes as a sum of residues, which are homology classes of the singular set. More precisely:

\begin{theorem}[Baum-Bott]\cite{Baum2} \index{Residue ! Baum-Bott}
Let $\mathcal{F}$ be a holomorphic foliation of dimension $k$ on a complex manifold $X$, and let $\varphi$ be a homogeneous symmetric polynomial of degree $d$ satisfying $n-k < d \leq n$. Let $Z$ be a compact connected component of the singular set $\mathrm{Sing}(\mathcal{F})$. Then, there exists a homology class $\mathrm{Res}_{\varphi}(\mathcal{F}, Z) \in \mathrm{H}_{2(n - d)}(Z; \mathbb{C})$ such that:
\begin{enumerate}
    \item[$(i)$] $\mathrm{Res}_{\varphi}(\mathcal{F}, Z)$ depends only on $\varphi$ and on the local behavior of the leaves of $\mathcal{F}$ near $Z$. 
    \item[$(ii)$] Suppose that $X$ is compact, and define $\mathrm{Res}({\varphi}, \mathcal{F}, Z) := \alpha_{\ast} \mathrm{Res}_{\varphi}(\mathcal{F}, Z)$,
    where $\alpha_{\ast}$ is the composition of the maps
    \[
    \mathrm{H}_{2(n - d)}(Z; \mathbb{C}) \stackrel{i_{\ast}}{\longrightarrow} \mathrm{H}_{2(n - d)}(X; \mathbb{C}) \quad \text{and} \quad \mathrm{H}_{2(n - d)}(X; \mathbb{C}) \stackrel{P}{\longrightarrow} \mathrm{H}^{2d}(X; \mathbb{C}),
    \]
   \noindent  where $i_{\ast}$ is the induced map from the inclusion $i : Z \longrightarrow X$, and $P$ is the Poincaré duality map. Then 
    \[
    \varphi (\mathcal{N}_{\mathcal{F}}) = \sum_{Z} \mathrm{Res}({\varphi}, \mathcal{F}, Z).
    \]
\end{enumerate}
\end{theorem}
\begin{figure}[h!] 
    \centering
    \includegraphics[width=0.6\textwidth]{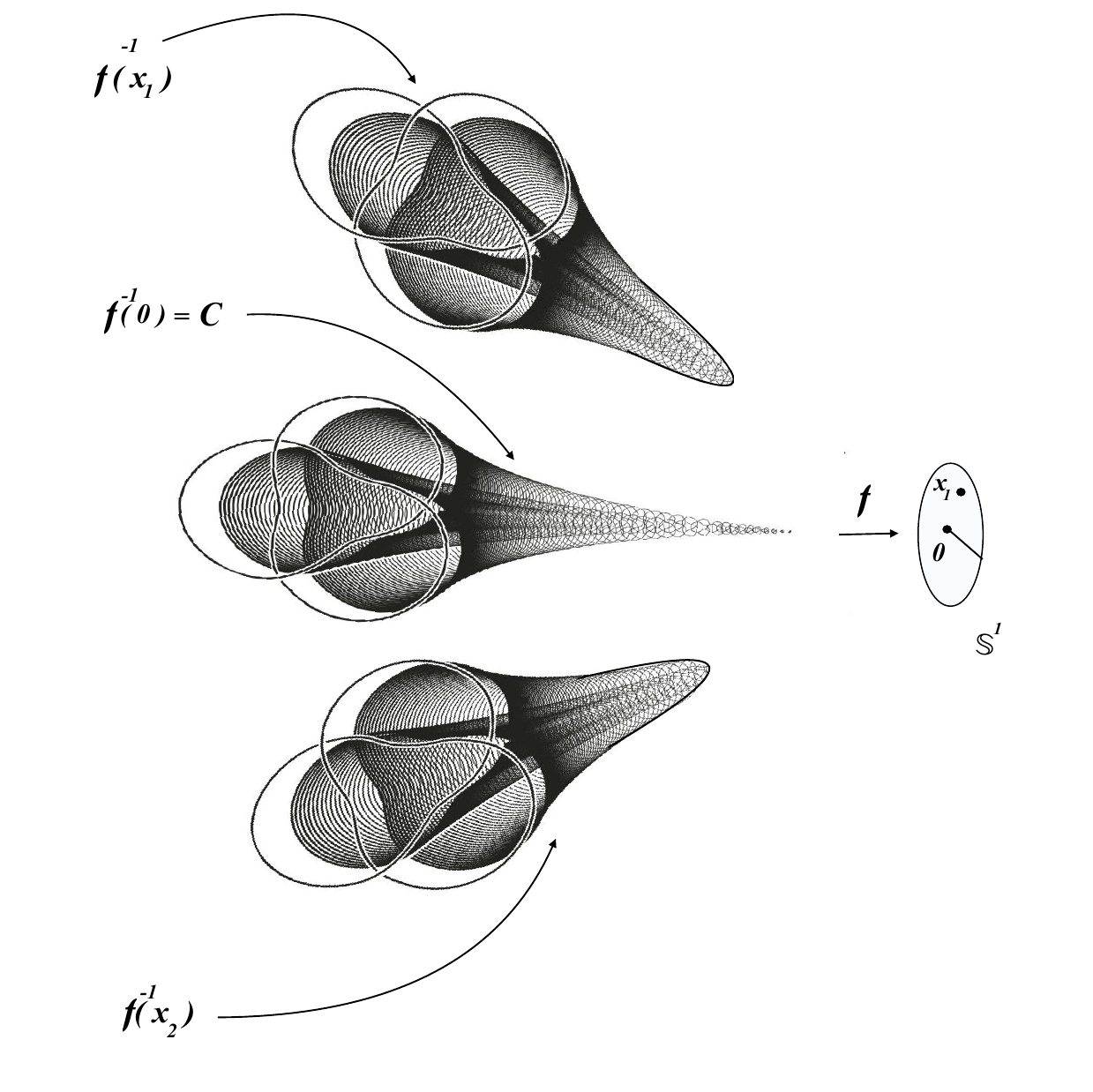}
    \caption{}
    \label{MF}
\end{figure}

If the foliation has dimension one and isolated singularities, the residues can be expressed in terms of  Grothendieck residues (see \cite{Baum1}).  More precisely, in this case, for each $p \in \operatorname{Sing}(\mathcal{F})$, an isolated singularity of $\mathcal{F}$, we have
\begin{eqnarray}\label{0001}
\operatorname{Res}(\varphi,\mathcal{F},p) = \operatorname{Res}_p\left[\begin{array}{cccc} \varphi(Jv)\\ v_1,\ldots,v_n\end{array}\right ],
\end{eqnarray}
where $v = (v_1, \ldots, v_n)$ is a germ of a holomorphic vector field on  $(U,p)$, serving as a local representative of $\mathcal{F}$,  $Jv$ is its Jacobian matrix, and 
$$
\operatorname{Res}_p\left[\begin{array}{cccc} \varphi(Jv)\\ v_1,\ldots,v_n\end{array}\right ]= \left(\frac{1}{2\pi \sqrt{-1}}\right)^n\int_{\Gamma}\frac{\varphi(Jv)}{v_1,\ldots,v_n},
$$
denotes  the Grothendieck residue of $\varphi(Jv)$ with respect to $v_1, \ldots, v_n$  at $p$, and $\Gamma$ is the cycle $\{x\in U; |v_i(x)|=\epsilon_i; \ i=1,\dots,n\}$.

In general, it is a difficult task to compute the Baum–Bott residues for foliations of higher dimension. 
In \cite{Vishik1973}, Vishik proved Baum–Bott's Theorem and provided a method to compute the residues, assuming that the foliation has a locally free tangent sheaf. 
In \cite{BracciSuwa2015}, F. Bracci and T. Suwa studied the behavior of Baum–Bott residues under smooth deformations.

Let us recall how to compute such residues in a particular case.
Let $\mathrm{Sing}_{n-k+1}(\sF)\subset \mathrm{Sing}(\sF)$ be  the subset of components of pure dimension $n-k+1$.
Consider a symmetric homogeneous polynomial  $\varphi$   of degree $n-k+1$. 
Let  $Z \subset \mathrm{Sing}_{n-k+1}(\sF)$ be an irreducible component.  Take a generic point $p \in 
Z$ such that $p$ is a point where $Z$ is smooth and disjoint from the other   irreducible 
components. Now, consider  $B_{p}$ a ball centered at $p $,  of dimension $n-k+1$ sufficiently small and transversal to $Z$ in $p$.
In \cite{Baum2}, under a generic condition( Kupka condition \cite{Rosas}), Baum and Bott proved that   
$$ \mathrm{Res}(\sF, \varphi ; Z) = \mathrm{Res}_{\varphi}(\sF|_{B_p}; p)[Z],  $$
where $ \mathrm{Res}_{\varphi}(\sF|_{B_p}; p)$ represents  the Grothendieck residue at $p$ of the one dimensional  foliation $\sF|_{B_p}$ on $B_p$ and $[Z]$ denotes  the fundamental cycle  associated to $Z$.

In \cite{CorreaLourenco2019}, the authors show that Baum–Bott's generic hypothesis is not necessary, and we demonstrate that the above formula always holds as long as the singular set of the foliation has dimension less than or equal to $n-k+1$. 
More recently, in \cite{Kaufmann2023}  Kaufmann, Lärkäng, and Wulcan showed that the Baum–Bott residues can be expressed in terms of residue currents.


\section{Indices for 1-dimensional holomorphic  foliations on surfaces}\label{S. Brunella}
In his remarkable paper \cite{Bru}, M. Brunella studied various invariants of holomorphic 1-dimensional foliations on complex surfaces that somehow spring from the Baum-Bott theory of residues. He then 
related all these invariants to establish interesting geometric properties of foliations. 
 These results have been  extended to higher dimensions and to other settings  by various authors.  This is the subject that we study in this section.

\subsection{Baum-Bott-Brunella indices}
In this subsection we look at the Milnor number of a foliation, already studied earlier (see \ref{Milnor number of a foliation}), as well as at 
the Camacho-Sad index \cite{CS}, the variation \cite{KS}, defined by  Khanedani-Suwa, and an invariant that Brunella called Baum-Bott invariant. 

Let $\F$ be a holomorphic foliation with isolated singularities on a complex
surface $X$ and let  $p$ be singular point  of   $\F$.  From the previous section we know that there are two  Baum-Bott  residues associated to $\F$ at $p$; these correspond to the Chern numbers  $c_2$ and $c_1^2$.
The one corresponding to  $c_2$, or to the symmetric function given by the determinant,  has as residue  at $p$  the local Poincaré-Hopf index of a vector field generating $\F$ locally. The local invariant we get is the local multiplicity, or Milnor number $\mu(\F,p)$ of the foliation defined in \ref{Milnor number of a foliation}, which is a topological invariant. Following  \cite{Bru} we denote this  invariant by $PH(\F, p) = \mu(\F,p)$.

The second  Baum-Bott residue corresponds to  the symmetric function given by the trace, or actually to the Chern number $c_1^2(N_\F)$, 
where $N_\F \in H^1(X,{\O}^*)$
is the normal bundle of $\F$;  this bundle is defined even in the presence of singularities since $\F$ is determined by a morphism of a line bundle into $TM$.

Following  \cite{Bru} we denote this by $BB(\F, p)$ and call it {\bf the Baum-Bott invariant of $\F$ at $p$.}  This can be studied  {\it à la}  Godbillon-Vey. Let   $ (z,w)$ be
 local coordinates centered at $p$ and let $\omega = F(z,w) dw    - 
G(z,w) dz$ be a
holomorphic 1-form that generates $\F$ locally. 
Let $\beta$ be a complex valued smooth 1-form of type $(1,0)$ on a punctured neighborhood $U^* = U \setminus \{p\}$
such that 
$$d \omega = \beta \wedge \omega \,.$$
For instance one may take:
$$\beta = \frac{\frac{\partial F}{\partial z} + \frac{\partial G}{\partial w}} {|F|^2 + |G|^2} \; \Big(\overline F \,dz  +  \overline G \,dw \Big)  \;.
$$
Then  one has (see  \cite{Bru}): 
\[\label{BB-inex}
 \mathrm{BB}(\F, p) \; = \, \frac{1}{(2 \pi \,i)^2} \int_{\s^3_\e} \beta \wedge d\beta \;,
\]
where ${\s^3_\e}$ is the boundary   sphere of a small
ball around $p$ with its natural orientation.

If the surface $X$ is compact then one has a global result in the vein of the Poincaré-Hopf index theorem:
\begin{equation}\index{Baum-Bott index}
c_1^2(N_\F) = \sum_{p \in {\rm sing} \F}  \mathrm{BB}(\F, p) \;.
\end{equation}

Now let $S$ be a separatrix of $\F$ at $p$.\index{Separatrix} That is, a local holomorphic curve containing $p$ which is invariant by $\F$. 
 It is not assumed that this curve is neither smooth nor irreducible. The following invariant is defined in \cite{KS}:
 
 \begin{definition}\index{Variational index}
 {\bf The variation} of $\F$ at $p$ relative to $S$ is:
 $${\rm Var}(\F,S,p) \,= \frac{1}{2 \pi i} \int_{S \cap \s_\e} \beta \,,
 $$
 with $\beta$ as above.
 \end{definition}

This invariant has  the following important global property. 
\begin{proposition}[Khanedani-Suwa]
Suppose $S$ is now a compact holomorphic curve in $X$, invariant by the foliation, then:
$$ \sum_{p \in {\rm Sing}(\F) \cap S} {\rm Var}(\F,S,p)  \, = \; c_1( N\F) \cdot S \;.
$$
\end{proposition}

Now we come to the Camacho-Sad index. 
This   plays a key role for proving in \cite {CS} the separatrix theorem: every germ of a holomorphic vector  field in $\C^2$ with an isolated singularity at the origen has at least one separatrix $S$. Yet, this index is not actually defined in $\C^2$ but up in a “desingularization” of the vector field, where one assumes that there is a smooth divisor $C$, invariant under the induced  foliation, with isolated singularities. One attaches a CS-index to each such singularity, which actually is a localization of the Chern class of the normal bundle of $C$. 

Let $S$ be again a separatrix in $X$ for the foliation $\F$ and $p \in S$ a singularity of the foliation. Let $\omega$ be a local 1-form on $X$ generating $\F$ locally and $f$  a local function that defines $S$ near $p$, and we assume it is reduced. Then we know from \cite{Suwa-surfaces} that there
are functions $g$, $\xi$ and a 1-form $\eta$ on a neighborhood of $p$ such that
$$g \omega = \xi \, df + f \eta \,$$
with $k \ne 0$ on $S^* = S \setminus \{p\}$. Notice that on $S$ we have $g \omega = \xi \,df$ and $g \ne 0$ on $S^*$.

\begin{definition} \index{Camacho-Sad}
The Camacho-Sad index of $\F$ at $p$ relative to $S$  is:
$$ \mathrm{CS}(\F, S, p) = \, - \frac{1}{2\pi i} \int_{S \cap \s_\e^3}  \frac{1}{\xi} \eta \,. 
$$
\end{definition}

One has: 

\begin{proposition}
$$
\sum_{p \in {\rm Sing}(\F) \cap S}  \mathrm{CS}(\F, S, p)  \, = \; S \cdot S  \, = \; c_1(N_S)[S] \;.
$$
\end{proposition}

Notice that all these formulas have a similar flavor, involving the geometry of the foliation to localize the Chern class of some normal bundle. This contrasts with the GSV index, 
which (by Proposition \ref{Th GSV vs virtual})  is the localization  of a Chern class of the virtual tangent bundle.

\subsection{The GSV index for 1-dimensional holomorphic  foliations}
Given a foliation $\F$ on a complex surface $X$ as above, with a singularity at $p$ and a separatrix $S$ through $p$, $S$ is locally generated by a holomorphic vector  field $v$. Hence $S$ can be regarded itself as being a hypersurface in $X$ with an isolated singularity at $p$ (in fact $p$ can be a regular point of $S$), equipped with a vector field $v$. Thus  one has the  local GSV index \index{GSV index ! for foliations} of $v$ in $S$ at $p$, which we take as being its Poincaré-Hopf index in a Milnor fiber of a local defining function $f$ for $S$ (we take this as its definition even if $S$ has several branches).

 Brunella  in \cite{Bru} used the foliation to introduce an invariant $Z(\F,S,p)$ in the setting we now envisage, and proved that it  coincides with the GSV index.  Brunella then established surprising deep relationships between the Baum-Bott, Camacho--Sad, the variation  and the $\operatorname{GSV}$ indices.

Let $X$ be a compact complex surface, and let $\mathcal{F}$ denote a one-dimensional holomorphic foliation on $X$. Let $C$ be a reduced curve on $X$ and $x$ a point in $C$.
When $C$ is invariant under $\mathcal{F}$, we say that $\mathcal{F}$ is {\it logarithmic along} $C$. 
Let
 $f = 0$ be a local equation for $C$ in a neighborhood $U_{\alpha}$ of $x$, and let $\omega_{\alpha}$ be a holomorphic $1$-form that induces the foliation $\mathcal{F}$ on $U_{\alpha}$. Since $\mathcal{F}$ is logarithmic along $C$, it follows from \cite{Sai,  Suwa-surfaces} that there exist holomorphic functions $g$ and $\xi$ defined in a neighborhood of $x$ that do not both vanish identically on $C$, such that
\begin{eqnarray} \label{2p11}
g\,\frac{\omega_{\alpha}}{f} = \xi\, \frac{df}{f} + \eta,
\end{eqnarray}
\noindent where $\eta$ is an appropriate holomorphic $1$-form.  

\begin{definition}[Brunella \cite{Bru}] \index{Brunella index}
 Let $\mathcal{F}$ be a one-dimensional holomorphic foliation on a compact complex surface $X$ that is logarithmic along a reduced curve $C \subset X$. Given $x \in C$, we define
$$
\mathrm{Z}(\mathcal{F}, C, x) = \sum_i \mathrm{ord}_x \left(\frac{\xi}{g}|_{C_i}\right),
$$
where $C_i \subset C$ are the irreducible components of $C$ at $x$, and $\mathrm{ord}_x \left(\frac{\xi}{g}|_{C_i}\right)$ denotes the order of vanishing of $\frac{\xi}{g}|_{C_i}$ at $x$.
\end{definition}


First, we will see that Brunella's $\operatorname{Z}$ index coincides with the $\operatorname{GSV}$ index.

\begin{proposition}\cite[Page 532 ]{Bru}
Let $p\in \operatorname{Sing}(\F)\cap C$. Then
\[
 \operatorname{Z}(\mathcal{F}, C, p)=\operatorname{GSV}(\mathcal{F}, C, p). 
\]
\end{proposition}
\begin{proof}
Let $f$ be a holomorphic
function on a neighborhood $U$ of $p$ and defining  $C|_U=\{f = 0\}$. So, if $\omega$ induces $\F$ on $U$ we have  

{
$$ 
\omega =  \frac{\xi}{g} df + \frac{f}{g}\eta.
$$
}

Let $J: \Omega_U^1 \to TU$ denote the isomorphism from $1$-forms to vector fields in a neighborhood $U$ of $p$, defined by
\[
\omega = P \, dx + Q \, dy \mapsto J(\omega) = Q \frac{\partial}{\partial x} - P \frac{\partial}{\partial y}.
\]
The holomorphic vector field $J(\omega)$ defines the foliation $\mathcal{F}$ and can be expressed as the sum of two meromorphic vector fields:
\[
J(\omega) = \frac{\xi}{g} J(df) + \frac{f}{g} J(\eta) = v_1 + v_2,
\]
with $v_1=\frac{\xi}{g} J(df)$ and $v_2= \frac{f}{g} J(\eta)$. Now, observe that 
$$J(\omega)|_{C}=\left(\frac{\xi}{g} J(df) + \frac{f}{g} J(\eta)\right)|_{C}=\frac{\xi}{g} J(df)|_{C}=v_1|_C$$
 and $v_1$ is tangent to $C_\epsilon = \{ f = \epsilon \}$,  for a small $\epsilon$. The restriction $v_1|_{C_\epsilon}$ has poles at the points of $C_\epsilon \cap \{ g = 0 \}$ and zeros at the points of $C_\epsilon \cap \{ \xi = 0 \}$. Since the GSV index in $C_\epsilon$ is the Poincar\'e--Hopf index of $v_1|_{C_\epsilon}$, it is the difference between the number of zeroes and the number of poles of $v_1|_{C_\epsilon}$. Therefore,
\[
\operatorname{GSV}(\mathcal{F}, C, p)=\frac{1}{2\pi i} \int_{\partial C_\epsilon} \frac{d\xi}{\xi} -  \int_{\partial C_\epsilon}  \frac{dg}{g} = \frac{1}{2\pi i} \int_{\partial C} \frac{g}{\xi} d \left(  \frac{\xi}{g}\right) = \operatorname{Z}(\mathcal{F}, C, p).
\] 
This proves the proposition.

\end{proof}


One has: 

\begin{proposition} \cite[ Proposition 5]{Bru}\label{var-CS-GSV}
If $C$ is any separatrix at $p$, then
\[
\operatorname{Var}(\mathcal{F}, C, p) = \operatorname{GSV}(\mathcal{F}, C, p) + \operatorname{CS}(\mathcal{F}, C, p),
\]
where $\operatorname{Var}(\F,V,p)$ denotes the variational Khanedani--Suwa index of $\F$ along $V$ at $p$ \cite{KS} and $\operatorname{CS}(\mathcal{F}, C, p)$ is the Camacho--Sad index \cite{CS}.
\end{proposition}

\begin{remark}
Recently, in \cite{CorreaFernandezSoares2021}, the authors generalized the variational index for currents invariant under foliations, with applications to the study of transcendental leaves of foliations and their accumulation on the singular set. 
\end{remark}

\begin{theorem}[Brunella \cite{Bru}] 
Let $\mathcal{F}$ be a one-dimensional holomorphic foliation on a compact complex surface $X$, logarithmic along a reduced curve $C \subset X$. Then
\begin{eqnarray}\nonumber
\sum_{x \in \operatorname{Sing}(\mathcal{F}) \cap C}  \operatorname{GSV}(\mathcal{F}, C, x) =  N\mathcal{F} \cdot C - C \cdot C\,,
\end{eqnarray}
\end{theorem}

The GSV index is an integer and we already know that it can be negative.  In fact, consider  Example  \ref{homogeneous} and the foliation in $\C^2$  given by the  homogeneous vector field  $v = w^k \frac{\partial}{\partial z} - z^k \frac{\partial}{\partial w} $, whose leaves are the fibers of the map $f = z^{k+1} + w^{k+1}$. Then $S= f^{-1}(0)$  is a separatrix (with $k+1$ branches) and its GSV index is $1 - k^2$, which is negative whenever $k >1$. Notice that this foliation  has infinitely many separatrices.

Marco Brunella remarked that
 the non-negativity of the GSV index 
 is the obstruction to the positive solution to “Poincaré problem”, as we explain below in Section \ref{Poincare Problem}. The following results from \cite{Bru} point in that direction and they are interesting on their own.

Recall that a 
separatrix $S$ of $\F$ at $p$ is nondicritical
if there is a sequence of blow-ups based at $p$, $\pi :  \widetilde X \to X$,    such that $\pi$ is  a resolution of $S$, $\widehat S = \pi^{-1}(S)$ has only normal
crossing singularities, and 
the divisor $\pi^{-1}(p)$ is  $\widehat F$-invariant, where $\widehat F$ is the induced foliation on $\widetilde X$ (see \cite{CLS}).
 Then (see   \cite[Propositions 6 and 7]{Bru}): 

\begin{theorem}
If $S$ is a nondicritical separatrix at $p$ then:
$$ \mathrm{GSV}(\F,S,p) \ge 0 \,.$$
If $S$ further is a generalized curve ({\it i.e.} there are no saddle-nodes in its resolution) and $S$  is maximal,  union of all separatrices at $p$, then:
$$ \mathrm{GSV}(\F,S,p) = 0 \,.$$
\end{theorem}

 In \cite{CavalierLehmann2001}, Cavalier and D. Lehmann proved that the converse is also true, i.e., if $$ \mathrm{GSV}(\F, S, p) = 0,$$ then the foliation $\F$ is a generalized curve at $p$, and $S$ is the maximal separatrix.

\section{Indices in higher dimensions and the Poincar\'e problem}\label{Poincare Problem}\index{Poincaré Problem}
As noted before, indices and residues play a significant role for the study of the famous Poincaré problem in dimensions 2 and higher. In this section we review some of these. The literature is vast and we will focus only on some of the works most related with this article.  Let us start by saying a few words about the problem itself.

In \cite {Poin} H. Poincaré asked  whether  it  is possible to decide if an
algebraic differential equation in two variables is algebraically integrable, and he  observed that in order to answer this question it is sufficient to know whether one can 
 bound  the degree of the generic leaves. He posed a question that in modern language can be stated as follows. Let $\F$ be a holomorphic foliation in $\mathbb P^2_{\C}$ and 
$C$ an irreducible   curve which is invariant by $\F$: Is it possible to bound the degree of $C$ in terms of the degree of $\F$?

This problem has been studied by many authors in more generality. One may ask, for instance, if  given a foliation $\F$ on $\mathbb P^n_{\C}$ of degree $d$  and a hypersurface $V$ in $\mathbb P^n_{\C}$  invariant by $\F$, is it possible to bound the degree of $V$ by a number   which depends only on $d$? 

The answer in general is obviously negative, even for $n=2$. For instance the foliation defined by the form $p w dz - q z dw$ has degree $1$ and the invariant curve $z^p - w^q = 0$ has degree equal to the maximum of $p, q$. There are many more easy examples. Yet, there are many positive examples, and as explained in \cite{BruMe} the philosophy   is that the failure of having a uniform bound   is due to the existence of "bad singularities" of $V$ or of $\F$. 

As mentioned before, for $n=2$ , Brunella in \cite{Bru} established a connection between the GSV-index and both the Khanedani-Suwa variational index \cite{KS} and the Camacho-Sad index \cite{CS}, and  he demonstrated that the non-negativity of the GSV-index acts as an obstruction to providing an affirmative answer to the \textit{Poincaré problem} on compact complex surfaces.

In fact, more generally, 
the non-negativity of the GSV-index provides an obstruction to solving the generalized Poincar\'e problem for Pfaff systems on projective complex spaces as explained below.

In \cite{Correa-MachadoGSV}, a GSV-type index for varieties was introduced, in the spirit of Brunella, for varieties invariant under holomorphic Pfaff systems (which may not be locally decomposable) on projective manifolds. 

Let $U$ be a germ of an $n$-dimensional complex manifold. Let $D$ be an analytic reduced hypersurface  on $U$ and 
consider its decomposition
into  irreducible components  
$$
D = D_1\cup\cdots \cup D_{ n-k},
$$
and suppose that   the analytic subvariety  $V = D_1\cap\cdots \cap D_{ n-k}$ has pure codimension $k$. We assume  that  
$$
V  = \{z\in U : f_{ 1}(z) =\cdots =f_{ n-k}(z) = 0\},
$$
 with $f_{1},\ldots,f_{  n-k}\in \mathcal{O}_U $ and for each $i \in \{1,\ldots, k \}$,
$$
D_i   = \{z\in U : f_{i}(z) = 0\}.
$$
\noindent Since $V$ is  a  reduced variety, then the   $k$-form $df_{1}\wedge \ldots \wedge df_{ n-k}$ is not identically zero on each irreducible component of $V$. 
Following Aleksandrov \cite{AlekMulti}, we denote by $\Omega^q_U(\hat{D}_i)$, for  $q\geq 1$, the $\mathcal{O}_U$-module of meromorphic differential $q$-forms  with simple poles on the $$\hat{D}_i = D_1\cup \cdots \cup  D_{i-1}\cup D_{i+1} \cup \cdots\cup D_{ n-k},$$
for each $i = 1,\ldots,n-k$. We have the following Aleksandrov's Decomposition Theorem, which is a generalization of Saito's residues \cite{Sai}.

\begin{theorem}[Aleksandrov \cite{AlekMulti}] \label{Teo_Alex}
Let $\omega  \in \Omega^q_U (D)$ be a meromorphic $q$-form with simple poles on $D$. If for each $j=1,\ldots, k$,
\begin{eqnarray}\nonumber
df_{ j}\wedge \omega  \in \displaystyle \sum_{i=1}^k\Omega^{q+1}_{U }(\hat{D}_i) 
\end{eqnarray}
\noindent then, there exist a holomorphic function $g $,   which is not identically zero on every irreducible component of $V$, a holomorphic $(q-k)$-form $\xi \in \Omega_{U }^{q-k}$ and a meromorphic $q$-form $\eta \in \sum^k_{i=1} \Omega^{q}_{U } (\hat{D}_i)$  
such that the following decomposition holds
\begin{eqnarray}\label{expr110}
g \omega  = \frac{df_{ 1}}{f_{ 1}}\wedge\cdots \wedge \frac{df_{ k}}{f_{ k}} \wedge  \xi  + \eta .
\end{eqnarray}
\end{theorem}

We can apply this decomposition in our context. 
\begin{proposition}\label{prop_dec}
    Let $\omega\in \mathrm H^0(X,\Omega_X^{n-k}\otimes  \mathcal{L})$ be  a Pfaff system of  dimension  $k$  on a complex manifold $X$, 
and  $V\subset X$ a reduced local complete intersection subvariety  of  dimension $n-k$ which is  invariant by $\omega$.
Then for all local representations  $\omega_{\alpha}= \omega|_{U_\alpha}$  of $\omega$,  and all local expressions  of $V$ in $U_\alpha$
$$
V\cap U_{\alpha} = \{z\in U_{\alpha}: f_{\alpha,1}(z)=\cdots= f_{\alpha,n-k}(z) = 0 \},
$$
 there  are  holomorphic functions $g_{\alpha}, \xi_{\alpha}\in \mathcal{O}(U_{\alpha})$  such that
\begin{eqnarray}\label{expr1110}
g_{\alpha}\, \omega_{\alpha} =   \xi_{\alpha} df_{\alpha,1}\wedge\cdots \wedge df_{\alpha,n-k}  + \eta_{\alpha}.
\end{eqnarray}
\noindent Moreover, $g_{\alpha}$  is not identically zero on every irreducible component of  $V$ and $\eta_{\alpha}$ is given by
\begin{eqnarray}\nonumber
\eta_{\alpha} = f_{\alpha,1}\,\eta_{\alpha,1}+\cdots+f_{\alpha,1}\,\eta_{\alpha,n-k},
\end{eqnarray}
\noindent where each $\eta_{\alpha,i}\in\Omega^{n-k}_{U_{\alpha}}$ is a holomorphic $(n-k)$-form.  In particular,  the restriction  $$\displaystyle \frac{\xi_{\alpha} }{g_{\alpha} }|_{V}$$  is well defined. 
\end{proposition}

 In this context, we define the GSV index in Brunella's spirit, which coincides with the GSV index when $V$ is a curve invariant under a foliation on a surface.

\begin{definition}  \index{GSV index ! Brunella’s definition} 
Under the assumptions of Proposition \ref{prop_dec}, suppose that $S \subset \mbox{Sing}(\omega) \cap V$ is a codimension one subvariety of $V$.  The GSV-index of $\omega$ relative to $V$ in $S$ is defined  by 
\begin{eqnarray}\nonumber
\mathrm{GSV}(\omega, V, S) : = \sum_j \mathrm{ord}_{S}\left(\displaystyle\frac{\xi_{\alpha}}{g_{\alpha}}|_{V_j}\right),
\end{eqnarray}
where the sum is taken over all irreducible components $V_j $ of $V$ and 
$\mathrm{ord}_{S}\left(\displaystyle\frac{\xi_{\alpha}}{g_{\alpha}}|_{V_j}\right)$ denotes  the order of vanishing of $\frac{\xi_{\alpha}}{g_{\alpha}}|_{V_j}$ along   $S$.
 \end{definition}

The following result was proved in \cite{Correa-MachadoGSV}.

\begin{theorem}\label{prop9}
Let $X$ be a projective manifold, and let $V \subset X$ be a reduced, locally complete intersection subvariety of  dimension $k$, invariant under a Pfaff system of  same dimension induced by a twisted form $\omega \in \mathrm{H}^0(X, \Omega_X^{n-k} \otimes \mathcal{L})$. Then the following statements hold:
\begin{itemize}
    \item[(a)] The integer  number $\mathrm{GSV}(\omega, V, S_i)$  depends only on the local representatives of $\omega$, $V$, and $S_i$;
    \\ 
    \item[(b)] If $\mathrm{Sing}(\omega, V) := \mathrm{Sing}(\omega) \cap V$ has codimension one in $V$, then the following formula is satisfied:
    $$
    \sum_i \mathrm{GSV}(\omega, V, S_i) [S_i] = c_1([\mathcal{L} \otimes \det (N_{V/X})^{-1}])|_V \frown [V],
    $$
    where $S_i$ denotes an irreducible component of $\mathrm{Sing}(\omega, V)$, and $N_{V/X}$ is the normal sheaf of the subvariety $V$.
\end{itemize}
\end{theorem}

The following result offers a method for calculating the GSV-index for Pfaff systems, which can be compared with  formulas in \cite[Proposition 5.1]{Suwa2014} and \cite{Gom}. Moreover, it partially addresses Suwa's question in \cite[Remark 5.3.(6)]{Suwa2014}.

\begin{theorem}\label{teo404} 
With the setup as in Theorem \ref{prop9}, let $x \in V$, and let 
$\{ f_{1} = \cdots = f_{n-k} = 0 \}$ 
be a local equation for $V$ in a neighborhood $U$ of $x$. Consider  
$$
\omega_{|U} = \sum_{\mid I \mid = n-k} a_I \, dZ_I, \,\,\,\,\, \text{where } a_I \in \mathcal{O}(U),
$$
as the holomorphic $(n-k)$-form that induces the Pfaff system $$\omega \in \mathrm{H}^0(U, \Omega_U^{n-k} \otimes \mathcal{L})$$ on $U$. Then the following formula holds:
\begin{eqnarray} \label{formula-gsv}
\mathrm{GSV}(\omega, V, S_i) = \mathrm{ord}_{S_i}(a_{I}|_V) - \mathrm{ord}_{S_i}(\Delta_{I}|_V),
\end{eqnarray}
\noindent where $\Delta_I$ is the $(n-k) \times (n-k)$ minor of the Jacobian matrix 
$\mathrm{Jac}(f_{1}, \ldots, f_{n-k})$, corresponding to the multi-index $I$.
\end{theorem}

More precisely, let $\omega \in \mathrm{H}^0(\mathbb{P}^n, \Omega_{\mathbb{P}^n}^k(d + n-k + 1))$ be a holomorphic Pfaff system of  dimension $k$ and degree $d$. Let $V \subset \mathbb{P}^n$ be a reduced complete intersection variety of dimension $k$ and multidegree $(d_1, \dots, d_{n-k})$, invariant under $\omega$. Suppose that $\mathrm{Sing}(\omega, V)$ has codimension one in $V$. Then,
\begin{eqnarray} \nonumber
\sum_i \mathrm{GSV}(\omega, V, S_i) \deg(S_i) = [d + n-k + 1 - (d_1 + \cdots + d_{n-k})] \cdot (d_1 \cdots d_{n-k}),
\end{eqnarray}
where $S_i$ denotes an irreducible component of $\mathrm{Sing}(\omega, V)$. Therefore, if $\mathrm{GSV}(\omega, V, S_i) \geq 0$ for all $i$, we have 
$$
d_1 + \cdots + d_{n-k} \leq d + {n-k} + 1.
$$

This bound is similar to those established by Esteves and Cruz \cite{EstCruz} and by Corrêa and Jardim \cite{CJ}. Related results were previously obtained by Cerveau and Lins Neto \cite{CLins},  Soares  \cite{SoaresPoinc1},  Cavalier and  Lehmann  \cite{CavalierLehmann2006},  Brunella and Esteves in \cite{BruMe}. In particular,  Soares in \cite{SoaresPoinc1} gave this bound for a foliation by curves:
$$
d_1 + \cdots + d_{n-1} \leq d + n,
$$
whenever $V$ is a smooth complete intersection invariant curve.  

\begin{example}
Consider the Pfaff system 
  $\omega \in \mathrm{H}^0(\mathbb{P}^n, \Omega_{\mathbb{P}^n}^{n-k}(d + n-k + 1))$
given by 
$$
\omega = \sum_{0 \le j \le +n-k} (-1)^j d_j f_j \, df_0 \wedge \dots \wedge \widehat{df_j} \wedge \dots \wedge df_{n-k},
$$
where $f_j$ is a homogeneous polynomial of degree $d_j$. We can see that
$$
d_0 + d_1 + \dots + d_{n-k} = d + n-k + 1.
$$
Suppose that $\deg(f_0) = d_0 = 1$ and that $V = \{ f_1 = \cdots = f_{n-k} = 0 \}$ is smooth. Then $V$ is invariant  by  $\omega$, and we have 
$$
d_1 + \dots + d_{n-k} = d + n-k.
$$
\end{example}

\subsection{ GSV-indices for vector fields  as residues}\label{ss. Suwa}

Suwa in \cite{Suwa2014} provided an alternative interpretation of the GSV and virtual indices for vector fields as residues on complete intersection varieties with isolated singularities. He considered the residue arising from the localization of the relevant Chern class of the ambient tangent bundle using a frame consisting of $v$ and additional vector fields. Suwa's definition is recalled as follows and does not depend on the choices made.

Let $U$ be a neighborhood of the origin $0$ in $\mathbb{C}^m$, and let $V$ be a subvariety (reduced, but not necessarily irreducible) of pure dimension $n$ in $U$. Assume that $V$ contains $0$ and that $V \setminus \{0\}$ is non-singular. We consider a closed ball $B$ around $0$ that is sufficiently small so that $R = B \cap V$ has a cone structure over $\partial R = L$, the link of $V$ at $0$. 

Let $E$ be a $C^{\infty}$ complex vector bundle of rank $l$ on $U$, and let $\textbf{s} = (s_1, \dots, s_r)$ be a $C^{\infty}$ $r$-frame of $E$ on a neighborhood $V_0$ of $L$ in $V$, where $r = l - n + 1$. Let $\nabla_0$ be an $\textbf{s}$-trivial connection for $E|_{V_0}$, and let $\nabla_1$ be a connection for $E$.

\begin{definition} \label{def2.1}\index{Residue !}
The residue of $\textbf{s}$ at $0$ with respect to $c_n$ is defined by
\[
\operatorname{Res}_{c_n}(\textbf{s}, E|_V ; 0) = \int_R c_n(\nabla_1) - \int_{\partial R} c_n(\nabla_0, \nabla_1).
\]
\end{definition}
We have an exact sequence
\begin{equation} \label{eq2.6}
0 \longrightarrow T V_0 \longrightarrow T U|_{V_0} \xrightarrow{\pi} N V_0 \longrightarrow 0,
\end{equation}
where $T V_0$ and $T U$ denote the holomorphic tangent bundles of $V_0$ and $U$, respectively, and $N V_0$ is the normal bundle of $V_0$ in $U$. Suppose that $V$ is a complete intersection defined by $f = (f_1, \dots, f_k)$ in $U$, where $k = m - n$. In a neighborhood of a regular point of $f$, we may choose $(f_1, \dots, f_k)$ as part of the local coordinates on $U$, so that we have holomorphic vector fields $\frac{\partial}{\partial f_1}, \dots, \frac{\partial}{\partial f_k}$ away from the critical set of $f$. 
These vector fields are linearly independent and normal to the non-singular part of $V$, so that 
\[
 \partial_{f}:= \left(\pi\left(\frac{\partial}{\partial f_1}\Big|_{V_0}\right), \dots, \pi\left(\frac{\partial}{\partial f_k}\Big|_{V_0}\right) \right)
\]
forms a frame for $N V_0$. Suppose that we have a $C^{\infty}$ nonsingular vector field $v$ on $V_0$. 

\begin{theorem} \cite[Theorems 3.4 and 4.4]{Suwa2014}
We have
\[
\mathrm{GSV}(v, 0) = \operatorname{Res}_{c_n}((v, \partial_{f}), T U|_V ; 0) = \mathrm{Vir}(v, 0),
\]
where $\mathrm{Vir}(v, 0)$ is the virtual index of $v$ at $0$. 
\end{theorem}

Since the indices do not depend on the choices made, we have the following consequence.
 \begin{corollary} Let  $V = \{ f_1 = \dots = f_k = 0 \}$ be a complete intersection, with isolated singularity at $0$,  invariant by a foliation by curves $\F$ singular at most at $0$. Then
\[
\mathrm{GSV}(\F, 0) = \operatorname{Res}_{c_n}(\F, T U|_V ; 0) = \mathrm{Vir}(\F, 0). 
\]
 \end{corollary}

Now, if $n=2$,  the formula \ref{formula-gsv}   generalizes  Suwa’s formula for the GSV index.    In fact, if $n=2$, then $k=1$, and $V$ is a curve, and $\omega$ induces a foliation by curves.  
Let $x \in V$, and 
$\{ f = 0 \}$ 
be a local equation for $V$ in a neighborhood $U$ of $x$. Consider  
$$
\omega =   a_1dx-a_2dy
$$
and its dual vector field
$$
v =   a_2  \frac{\partial}{\partial x}    +a_1  \frac{\partial}{\partial y}
$$

\begin{proposition}\cite[Corollary 5.2]{Suwa2014}
    Let  $V = \{ f = 0 \}$ be a curve, with isolated singularity at $0$,  invariant by a foliation by curves $\F$ singular at most at $0$. Then
\begin{eqnarray*} 
\mathrm{Z}(\omega, V, x) & =& \mathrm{ord}_{x}(a_{1}|_V) - \mathrm{ord}_{x}(\partial_{x}f|_V)=\mathrm{ord}_{x}(a_{2}|_V) - \mathrm{ord}_{x}(\partial_{y}f|_V).
\end{eqnarray*}
\end{proposition}

An interesting application is that the $\mathrm{GSV}$ index provides an obstruction to the continuous extension of the normal bundle of the foliation from the regular part of the invariant variety. Indeed,  let $(V, 0)\subset \mathbb{C}^m$ be an ICIS of dimension $n \geq 2$, and let $v$ be a holomorphic vector field on $V$, singular only at 0. This vector field defines a one-dimensional holomorphic foliation $\F$ on $V$ singular at $0$. On $V \setminus \{0\}$, we have the tangent bundle $T\F$ to the foliation, and we consider   the smooth normal bundle $N\F=T(V \setminus \{0\}) / T\F$.

\begin{corollary}\cite[Theorem 2]{Seade2008} 
\label{cor:normal-bundle-obstruction}
Let $(V, 0)$ and $\F$ be as above. Then the normal bundle to $\F$ in $V \setminus \{0\}$ extends to $0$ as a vector bundle (continuous or smooth) if and only if $\mathrm{GSV}(\F, 0)$ is a multiple of $(n - 1)!$
\end{corollary}


\section{ Log Baum--Bott residues and Aleksandrov's logarithmic index}\label{Log-Alex}\index{logarithmic index}

Using techniques from homological algebra, Aleksandrov \cite{Alex} introduced an algebraic index, inspired by the work of G\'omez-Mont, to measure the variation between the Poincar\'e--Hopf and homological indices.
 This is defined as follows: Let $v$ be a holomorphic local vector field that induces foliation $\mathcal{F}$. The interior multiplication $i_{v}$ then generates a complex of logarithmic differential forms:
$$
0 \longrightarrow \Omega^n_{X,p}(\log \, \mathcal{D}) \stackrel{i_v}{\longrightarrow} \Omega^{n-1}_{X,p}(\log \, \mathcal{D}) \stackrel{i_v}{\longrightarrow} \cdots \stackrel{i_v}{\longrightarrow} \Omega^{1}_{X,p}(\log \, \mathcal{D}) \stackrel{i_v}{\longrightarrow} \mathcal{O}_{X,p} \longrightarrow 0.
$$
Let $p$ be an isolated singularity of the foliation $\mathcal{F}$. The $i_v$-homology groups of the complex $\Omega^{\bullet}_{X,p}(\log \, \mathcal{D})$ are finite-dimensional vector spaces. Accordingly, the Euler characteristic
$$
\chi (\Omega^{\bullet}_{X,p}(\log \, \mathcal{D}), i_v) = \sum_{i=0}^n (-1)^i \dim H_i(\Omega^{\bullet}_{X,p}(\log \, \mathcal{D}), i_v)
$$
of this complex of logarithmic differential forms is well-defined. Since this invariant does not depend on the choice of the local representative $v$ of the foliation $\mathcal{F}$ at $p$, the {\it logarithmic index} of $\mathcal{F}$ at the point $p$ is given by
$$
\operatorname{{Ind}_{\log}}(\mathcal{F}, \mathcal{D}, p) := \chi (\Omega^{\bullet}_{X,p}(\log \, \mathcal{D}), i_v).
$$

In \cite{CorreaMachado2019, CorreaMachado2024}, the authors addressed the problem of establishing a global logarithmic residue theorem for meromorphic vector fields on compact complex manifolds. They proved the following result under the assumption that the invariant divisors are of normal crossings type.

\begin{theorem} 
Let $\mathcal{F}$ be a one-dimensional foliation with isolated singularities and logarithmic along  a normal crossing divisor $D$ in a compact complex manifold $X$. Then
$$
\int_{X} c_{n}(T_{X}(-\log D) - T\mathcal{F}) = \sum_{x \in \operatorname{Sing}(\mathcal{F}) \cap (X \setminus D)} \mu_x(\mathcal{F}) + \sum_{x \in \operatorname{Sing}(\mathcal{F}) \cap D} \operatorname{Ind}_{\log}(\mathcal{F}, D, x),
$$
where  $\mu_x(\mathcal{F})$ is the Milnor number of $\mathcal{F}$ at $x$. 
\end{theorem}

In \cite{CorreaMachado2019}, the authors prove that if $\mathcal{F}$  is  a one-dimensional foliation on $\mathbb{P}^n$, with $n$ odd, and  isolated singularities  (non-degenerate)  along a  smooth divisor $D$, then  $\mbox{Sing}(\F)\subset D$ if and only if  Soares’ bound \cite{SoaresPoinc1} for
the Poincar\'e problem is achieved, i.e, if and only if
$$
\deg(D)=\deg(\F)+1.
$$

More recently, in \cite{CLM2024}, the authors developed a Baum–Bott residue theory for foliations by logarithmic curves along free divisors:

\begin{theorem}\label{logBB}
    Let $\mathcal{F}$ be a one-dimensional holomorphic foliation on a complex manifold $X$, logarithmic along a free divisor $D \subset X$, and let $\varphi$ be a homogeneous symmetric polynomial of degree $n$. If $\mathcal{F}$ has  only isolated singularities, then  
$$
\int_{X} \varphi(TX(-\log D) - T\mathcal{F}) = \sum_{x \in \operatorname{Sing}(\mathcal{F}) \cap (X \setminus D)} \operatorname{BB}_{\varphi}(v, x) + \sum_{x \in \operatorname{Sing}(\mathcal{F}) \cap D} \operatorname{Res}^{\log}_{\varphi}(\mathcal{F}, p),
$$
where $\operatorname{BB}_{\varphi}(v, x)$ is the Baum--Bott residue of $\mathcal{F}$ at $x$ with respect to $\varphi$.  
If, in addition, $D$ has normal crossings at $p$, then
$$
\operatorname{Res}^{\log}_{\varphi}(\mathcal{F}, D, p) = (2\pi\sqrt{-1})^n \operatorname{Res}_p\left[\begin{array}{cccc}\varphi(J_{\log} \vartheta)  \\ v_1, \ldots, v_n \end{array}\right].
$$
Also, when $\varphi = \det$, we have 
$$
\operatorname{Res}^{\log}_{\det}(\mathcal{F}, D, p) = (-1)^n (2\pi\sqrt{-1})^n \operatorname{Ind}_{\log}(\mathcal{F}, D, p).
$$
\end{theorem} 

As an application of this result, we can derive a Baum–Bott-type formula for singular varieties. Consider $T_Y := \operatorname{Hom}(\Omega_Y^1, \mathcal{O}_Y)$ the tangent sheaf of $Y$, where $\Omega_Y^1$ represents the sheaf of Kähler differentials on $Y$. A holomorphic foliation on $Y$ is a reflexive subsheaf $T\F$ of $T_Y$ such that $[T\F,T\F]\subset T\F$. Then
the following result holds.

\begin{corollary}
     \label{Baum-Bott-sing}
Let $Y$ be a compact normal variety, and let $T\mathcal{F} \subset T_Y$ be a one-dimensional  holomorphic foliation with isolated singularities. Let $\varphi$ be a homogeneous symmetric polynomial of degree $n = \dim(Y)$. If $\pi: (X, D) \to (Y, \emptyset)$ is a functorial log resolution of $Y$, then  
$$
\int_{X} \varphi(T_{X}(-\log D) - (T \pi^{-1}\mathcal{F})^*) = \sum_{x \in \operatorname{Sing}(\mathcal{F}) \cap Y_{\operatorname{reg}}} \operatorname{BB}_{\varphi}(\mathcal{F}, x) + \sum_{S_{\lambda} \subset D \cap \operatorname{Sing}(\pi^* \mathcal{F})} \operatorname{Res}^{\log}_{\varphi}(\pi^* \mathcal{F}, S_{\lambda}),
$$
where $\operatorname{BB}_{\varphi}(\mathcal{F}, x)$ is the Baum–Bott index of $\mathcal{F}$ at $x$ with respect to $\varphi$. In particular, if $D \cap \operatorname{Sing}(\pi^{-1} \mathcal{F})$ is isolated, then  
$$
\int_{X} \varphi(T_{X}(-\log D) - (T\pi^{-1} \mathcal{F})^*) =  \qquad \qquad \qquad \qquad \qquad  \qquad \qquad \qquad \qquad $$ $$ \qquad \qquad \qquad  =  \sum_{x \in \operatorname{Sing}(\mathcal{F}) \cap Y_{\operatorname{reg}}} \operatorname{BB}_{\varphi}(\mathcal{F}, x) + \sum_{x \in D \cap \operatorname{Sing}(\pi^{-1} v)} \operatorname{Res}_p\left[\begin{array}{cccc} \varphi(J_{\log}(\pi^{-1} v)) \\ v_1, \ldots, v_n \end{array}\right].
$$
If, in addition, \(  \pi^{-1} \mathcal{F} \) is induced by a global vector field $v$ and \( \operatorname{Sing}(\mathcal{F}) \subset \operatorname{Sing}(Y) \),  then the following  \textit{Poincaré–Hopf}   type formula  holds: 
$$
\chi(Y) = \sum_{x \in D \cap \operatorname{Sing}(\pi^{-1} v)} \operatorname{Res}_p\left[\begin{array}{cccc} \varphi(J_{\log}(\pi^{-1} v)) \\ v_1, \ldots, v_n \end{array}\right] + \chi(\operatorname{Sing}(Y)).
$$
\end{corollary}

  As a consequence of Theorem \ref{logBB}, we can show that the difference between the Baum-Bott and logarithmic indices can be expressed in terms of the GSV and Camacho-Sad indices.
Before stating this result, let us first establish some notation.  If  $S\subset X$  is an analytic subset and $\F$ is a foliation by curves in $X$, with isolated singularities, we denote by
$$
\mathrm{R}(\F,S)=\sum_{x\in  Sing(\F)\cap S}  \mathrm{R}(\F,x),
$$
where $\mathrm{R}$ denotes some residue associated with $\F$ along $S$.    One gets (see \cite[Corollary 1.5]{CLM2024}:

\begin{theorem} 
Let $\F$ be a one-dimensional holomorphic foliation on a compact complex surface $X$ with a  reduced  invariant curve $S \subset X$. Then
$$
\BB_{c_1^2}(\F, S)-\Res^{\log }_{c_1^2}(\F, S) =2\GSV(\F,S)-\CS(\F,S).
$$
Hence, in particular,  the  non-negativity  of 
$$
\BB_{c_1^2}(\F, S)-\Res^{\log }_{c_1^2}(\F, S) +\CS(\F,S)
$$
is an  obstruction to an affirmative answer to the Poincar\'e problem.  Moreover, if $\F$  is a generalized curve along $S$, then
$$
\Res^{\log }_{c_1^2}(\F, S)=2\Res_{c_1^2}(\F, S).
$$
\end{theorem}

Finally, a new obstruction to giving an affirmative answer to the Poincar\'e problem for foliations by curves on $\mathbb{P}^n$,
$n$ odd, is introduced as a log index as follows(see \cite[Corollary 1.4]{CLM2024}):
let $\F$ be a one-dimensional foliation, with isolated singularities, on $\mathbb{P}^n$ with an invariant \textit{free divisor} $D$. If  $n$ is odd   and 
    $$
 \Res^{\log }_{c_{1}^n}(\F, \mbox{Sing}(\F))\geq 0
$$
then 
\begin{eqnarray*} 
\deg(D) \leq  \deg(\F) +n.
\end{eqnarray*}

\section{Final remarks}\label{s. final}

We have explored  the intertwining amongst invariants for singular varieties  and  for singular foliations. There is a lot more in the literature, and in this section we briefly comment on some recent developments related to this work.

The Milnor number permeates  this work and it is a topological invariant, fundamental  in singularity theory. There are plenty of  extensions and applications of this concept in the literature, see for instance \cite{CMS-2} for an account on this.

The  Tjurina number was introduced by Greuel in \cite{Gr2}; this is 
 the
dimension of the base space of a semi-universal deformation of an isolated  hypersurface singularity. In fact one has the Milnor and the Tjurina algebras. For a map germ $(\C^n,o) \to (\C,0)$ these are, respectively:
  $$M_f \, =  \,   \frac {\O_{\C^n,0}} { ( \frac{\partial f}{\partial z_1}, \cdots  \frac {\partial f}{\partial z_n} )} \quad \hbox{and} \quad 
 \mathcal T_f \, =  \,  \frac {\O_{\C^n,0}} { (f, \frac{\partial f}{\partial z_1}, \cdots \frac {\partial f}{\partial z_n} )}  \,.
 $$

 Their dimensions over $\C$ are the corresponding Milnor and Tjurina numbers, $\mu(f)$ and $\tau(f)$.
 There are deep relations between these two invariants (see for instance  \cite{Greuel,Dim-Gr}).

The Tjurina algebra  for holomorphic vector fields, to our knowledge, appears first in Gómez-Mont’s formula  in \cite {Gom} for the homological index (see \ref{Tjurina number of a vector field}), though  he does not use this name. The actual name of Tjurina number for foliations was coined in \cite[p. 159]{CaCoMo}, where 
F. Cano, N. Corral and R. Mol   use the intersection properties of polar curves for 1-dimensional  foliations in complex surfaces, to study properties about  non-dicritical singularities. In particular, they use  the formula in \cite {Gom} to give a polar interpretation of the GSV-index and  use this to study 
the  Poincaré problem for foliations in  $\mathbb P^2_\C$. 

In \cite {FGS} the authors study relations between the Milnor and Tjurina numbers
of singular foliations  in the complex plane with respect to a balanced divisor
of separatrices, an interesting concept introduced in \cite{Gen}. In the non-dicritical case, where there are only finitely many separatrices of the foliation, a balanced divisor just means 
the  union of all separatrices. For dicritical singularities this is 
  a finite set of separatrices which includes  all the isolated separatrices and some separatrices
from the ones associated to dicritical components. A balanced divisor  
provides a certain control of the algebraic multiplicity of the foliation and of its separatrices. The authors find in \cite {FGS} 
 important relations between various indices of a
singular foliation, 
  such as  the Baum-Bott, Camacho-Sad, GSV variational indices   and the Milnor and Tjurina numbers.

In \cite{CLS2}, Camacho, Lins Neto and Sad posed the problem of determining whether non-trivial minimal sets exist for foliations on \( \mathbb{P}^2 \). This question remains an open problem to this day. In \cite{LinsN}, Lins Neto proved the non-existence of minimal sets for codimension one foliations on $\mathbb{P}^n$, with $n \geq 3$. Motivated by Lins Neto's work, Marco Brunella in \cite{Brunella2008} initiated the investigation of this problem in higher dimensions for codimension one holomorphic foliations with ample normal bundles. 
More precisely, Brunella proposed the following conjecture: \textit{let $\mathcal{F}$ be a codimension-one foliation on a projective manifold of dimension $n \geq 3$. If $\mathcal{F}$ possesses an ample normal bundle, then every leaf of $\mathcal{F}$ necessarily accumulates on its singular set.} In the remarkable work \cite{Brunella-Perrone} by Brunella and Perrone, the authors, utilizing Baum-Bott residues, proved this conjecture  for projective  manifolds  with Picard rank one.  In \cite{AdachiBrinkschulte}, Adachi and Brinkschulte completely solved Brunella's conjecture by employing the localization of the first Atiyah class, folowing \cite{AbateBracciSuwaTovena},  of the determinant of the normal bundle of the foliation. A second proof of Brunella's conjecture was provided by Adachi, Biard, and Brinkschulte in \cite{AdachiBiardBrinkschulte2023}, where the authors employed a residue formula for meromorphic connections. This approach was inspired by a residue formula  proved  by Pereira in \cite{Pereira2024}.  
Recently, in \cite{CorreaFernandezSoares2021}, the authors generalized the variational index for currents invariant under foliations, with applications to the study of transcendental leaves of foliations and their accumulation on the singular set. This result represents a higher-dimensional generalization of the result previously established by Brunella for foliations by curves, as proven in \cite{Brunella2006}.

 The adjunction formulas for foliations have played an important role in the classification of regular foliations on  projective surfaces, as given by Brunella, and also in foliated birational geometry, see \cite{Brunella2004}. When the curve is invariant under the foliation, the difference $K_\mathcal{F} - K_C$ is the total sum of the GSV indices  along $C$. More recently, Cascini and Spicer \cite{CasciniSpicer2024} obtained an adjunction formula for foliations on algebraic varieties, where a correction term serves as a generalization of Brunella's GSV index whenever the divisor is invariant under the foliation.  We  observe that Theorem \ref{prop9} also provides an adjunction formula
$$
K\F|_V-K_V=  \det(N\F)|_V + \det (N_{V/X})^{-1}= \sum_i \mathrm{GSV}(\omega, V, S_i) [S_i], 
$$ 
since $\det(N\F)|_V= -K_X|_V +K\F|_V$ and $ \det(N_{V/X})=-K_X|_V+K_V$.
 
In fact the GSV index was originally defined for vector fields on hypersurface singularities. The extension to complete intersections was given in  \cite{SS}  where this is used to 
prove an  adjunction formula for  local complete intersections.


\printindex

\end{document}